\documentclass{amsart}

\usepackage{xcolor}
\usepackage{amsmath}
\usepackage{ amssymb}
\usepackage{comment}

\allowdisplaybreaks

\title{Stochastic integrals and Gelfand integration in Fr\'echet spaces}
\author{Fred Espen Benth and Luca Galimberti}

\newtheorem{theorem}{Theorem}[section]
\newtheorem{definition}[theorem]{Definition}
\newtheorem{lemma}[theorem]{Lemma}
\newtheorem{proposition}[theorem]{Proposition}

\newtheorem{corollary}[theorem]{Corollary}
\newtheorem{example}[theorem]{Example}

\newtheorem{remark}[theorem]{Remark}
\newcommand{\R}{\mathbb R}
\newcommand{\C}{\mathbb C}
\newcommand{\N}{\mathbb N}
\newcommand{\norm}[1]{\left\lVert#1\right\rVert}
\newcommand{\abs}[1]{\left |#1\right|}

\date{\today}

\thanks{Luca Galimberti has been supported in part by the grant {\it Waves and
    Nonlinear Phenomena (WaNP)} from the Research Council of Norway.}

\begin{document}

\maketitle

\begin{abstract}
We provide a detailed analysis of the Gelfand integral on Fr\'echet spaces, showing among other things a Vitali threorem, dominated convergence and a Fubini result. Furthermore, the Gelfand integral commutes with linear operators. The Skorohod integral is conveniently expressed in terms of a Gelfand integral on Hida distribution space, which forms our prime motivation and example. We extend several results of Skorohod integrals to a general class of pathwise Gelfand integrals. For example, we provide generalizations of the Hida-Malliavin derivative and extend the {\it integration-by-parts} formula in Malliavin Calculus. A Fubini-result is also shown, based on the commutative property of Gelfand integrals with linear operators. Finally, our studies give the motivation for two existing definitions of stochastic Volterra integration in Hida space. 
\end{abstract}

\section{Introduction}
We study the Gelfand integral on Fr\'echet spaces with the aim at providing a general pathwise approach to stochastic integration. Our theoretical considerations and analysis significantly generalize the theory of stochastic integration in the White Noise analysis context. We extend several known results, as well as develop new results.

The leading example and motivation for our studies is stochastic integration with respect to Brownian motion cast in White Noise analysis. Following Hida {\it et al.} \cite{HKPS}, let $(\mathcal S)$ be the Hida test function space, which is the subspace of smooth random variables of $L^2(\mathbb P)$, with $\mathbb P$ being the white noise probability measure on the Schwartz distributions $\mathcal S'(\mathbb R)$. On the dual $(\mathcal S)'$, the Hida distribution space, one can define the stochastic integral with respect to Brownian motion $B$ for processes $X$ in the Hida distribution space as the so-called Skorohod integral,
\begin{equation}
\label{def:skorohod-integral}
    \int_0^tX(s)\delta B(s):=\int_0^t X(s)\diamond W(s)ds.
\end{equation}
Here, $W$ is the {\it white noise} given by the time-derivative of Brownian motion $W(t):=\dot{B}(t)$, $\diamond$ is the Wick product defined on the Hida distribution space and the right-hand side in \eqref{def:skorohod-integral} is the Pettis integral in $(\mathcal S)'$. 
Thus, we see from the definition \eqref{def:skorohod-integral} of the Skorohod integral  that we can do integration with respect to a function (here Brownian motion $B$) in terms of a Pettis integral with respect to a measure on a measure space (here the Lebesgue measure on the real line). To achieve this, we must introduce an operator acting on the integrand (here Wick multiplication) which depends on the path of the integrator (here the derivative of the Brownian motion, $W$). Moreover, the integral is set in the dual of a Fr\'echet space (here the Hida distribution space).

With this in mind, we can approach stochastic integration abstractly as Gelfand integration of Fr\'echet-valued integrands over general measure spaces. 
In this paper we provide an extensive analysis of Gelfand integration on Fr\'echet spaces, i.e., on integrals of the type 
\begin{equation}
\label{gelfand-int-intro}
    \int_E\psi(x)\mu(dx)
\end{equation}
where $\mu$ is a measure on a measurable space $(M,\mathcal M)$ and $\psi:M\rightarrow\mathfrak Z'$, the topological dual of a Fr\'echet space $\mathfrak Z$.

The Gelfand integral \eqref{gelfand-int-intro} lends itself to several nice properties, as stability in terms of continuity with respect to both the the domain of integration and the integrands. The stability with respect to the integration domain follows from playing with the topology on the Fr\'echet space and the measure space. We show a Vitali theorem and a dominated convergence result, where the arguments rest on analysis of the weak topology on $\mathfrak Z'$. We are also able to prove a rather general   
Fubini-type result for multiple Gelfand integrals. A central result for our further analysis is the linearity of the Gelfand integral, in the sense that linear operators commute with the integral in \eqref{gelfand-int-intro}. In our proceedings, different topologies on the dual space $\mathfrak Z'$ are intensively discussed and analysed in relevant contexts.  

We apply our general Gelfand integration framework to study pathwise integration by considering integrators which may be represented as integral operators on Fr\'echet spaces. I.e., we are interested in Gelfand integrals like
$$
\int_E\gamma(x)d\xi(x)
$$
where we can associate a map $\Xi$ on $\mathfrak Z'$ such that 
$$
\int_E\gamma(x)d\xi(x)=\int_E\Xi(\gamma)(x)\mu(dx)
$$
This setting includes Skorohod integrals like \eqref{def:skorohod-integral} in Hida space. Our abstract theory for Gelfand integrals enable us to establish several general results for the pathwise integrals as well as some new proofs for the Hida case. From the topological properties of Fr\'echet spaces and their duals, we are able to generalize the notion of the Hida-Malliavin derivative. Indeed, under some restrictions on $\Xi$, we can define an object being the dual operator of $\Xi$. If this dual is a derivative operator, it is a true generalization of the Gross derivative in Benth \cite{Benth}. Moreover, it provides us with a significant extension of the well-known "integration-by-parts formula" in Malliavin Calculus (see Nualart, \cite{Nualart-book} say), as well as the connection between Stratonovich- and Ito-type stochastic integrals. Our results in this direction are based on Fr\'echet spaces with an algebra structure. Furthermore, general Fubini-type results for multiple pathwise integrals are provided, as well as stability results of the integral with respect to both integrands and integrators. As a final application, we discuss stochastic Volterra integrals in the Hida distribution space, where we relate different definitions to the choice of map $\Xi$. Our study provides a motivation for two existing definitions of stochastic Volterra integration as found in Alos {\it et al.} \cite{AMN} and Barndorff-Nielsen {\it et al.} \cite{BNBPV}.   

Our analysis and derivations are presented as follows: in Section 2 we provide some background material on Fr\'echet spaces, including some further discussions on the particular case of Hida spaces. The analysis of the general Gelfand integral is presented in Section 3, whereas the pathwise integration theory can be found in Section 4. In an appendix, we prove some results related to white noise analysis and generalized Malliavin derivative which is needed in Section 4.


\section{A rapid excursion on Fr\'echet spaces}\label{sec: rapid excursion}

A formal prerequisite for the reading of this paper is familiarity with the basic facts of the theory of locally convex spaces. The purpose of this preliminary section is not to establish these facts, but to clarify terminology and notation, and recall the results which will be used in the sequel. The reference we are sourcing from is the monumental book by Schaefer \cite{Schaefer}.

Let $\mathbb{F}\in\{\R,\C\}$. For an arbitrary topological vector space $(\mathfrak{X},\mathfrak{T})$ over $\mathbb{F}$, $\mathfrak{X}^\ast$ will denote its algebraic dual, while $\mathfrak{X}'$ will denote its topological dual, which is clearly a subspace of $\mathfrak{X}^\ast$.

A topological vector space $\mathfrak{Z}$ over $\mathbb{F}$ will be termed a \texttt{Fr\'echet space} if it is metrizable, complete and locally convex. Being it Fr\'echet, we may assume that there exists a non decreasing sequence of seminorms $(p_n)_{n\in \N}$ which generates its topology (see \cite[page 48]{Schaefer}).

If $\mathfrak{Z}'$ designates its topological dual, then $\mathfrak{Z}'$ separates points (see \cite[Page 48]{Schaefer}). This fact allows us to deduce that $\mathfrak{Z}$ and $\mathfrak{Z}'$ are in duality (see \cite[Page 123]{Schaefer}). Depending on the situation, we will denote the duality either $\langle \mathfrak{Z},\mathfrak{Z}' \rangle$ or $\langle \mathfrak{Z}',\mathfrak{Z} \rangle$.

We will consider different topologies at once on $\mathfrak{Z}'$. Let us list the ones we will be using. We recall that a locally convex topology $\mathfrak{T}$ on $\mathfrak{Z}'$ is called \texttt{consistent} with the duality $\langle \mathfrak{Z}',\mathfrak{Z} \rangle$ if $ (\mathfrak{Z}',\mathfrak{T})' =\mathfrak{Z}  $, where $\mathfrak{Z}$ is viewed as a subspace of $(\mathfrak{Z}')^\ast$. All consistent topologies are Hausdorff.

\begin{enumerate}
    \item The \texttt{weak star topology} $\sigma(\mathfrak{Z}',\mathfrak{Z})$, which is by definition the coarsest consistent topology on $\mathfrak{Z}'$. $\mathfrak{Z}'$ endowed with $\sigma(\mathfrak{Z}',\mathfrak{Z})$ is called the \texttt{weak dual} of $\mathfrak{Z}$.
    \item The \texttt{strong topology} $\beta(\mathfrak{Z}',\mathfrak{Z})$, which is the topology induced by the following seminorms (see \cite[page 81]{Schaefer}):
    \begin{equation*}
    p_B(z') = \sup_{z\in B}\abs{\langle z',z\rangle}\,, \;\;\;z' \in \mathfrak{Z}'\,,
    \end{equation*}
    where $B\subset\mathfrak{Z}$ ranges over all bounded (or, equivalently, weakly bounded, see Schaefer \cite[page 132]{Schaefer}) subsets of $\mathfrak{Z}$. $\mathfrak{Z}'$ endowed with $\beta(\mathfrak{Z}',\mathfrak{Z})$ is termed the \texttt{strong dual} of $\mathfrak{Z}$. (It is worth to stress that in general the strong topology $\beta(\mathfrak{Z}',\mathfrak{Z})$ is not consistent with the duality, \cite[page 141]{Schaefer}). We recall than an arbitrary net $(z'_i)_{i}\subset\mathfrak{Z}'$ converges to $z'\in\mathfrak{Z}'$ with respect to $\beta(\mathfrak{Z}',\mathfrak{Z})$ if and only if $p_B(z'_i-z')\to 0$ for any $B\subset\mathfrak{Z}$ bounded.
    \item The \texttt{Mackey topology} $\tau(\mathfrak{Z}',\mathfrak{Z})$, which is by definition the finest consistent topology on $\mathfrak{Z}'$ (see \cite[Corollary 1, page 131]{Schaefer}).
\end{enumerate}

The following useful facts will be used repeatedly and tacitly during the proofs of the subsequent results:
\begin{enumerate}
    \item A mapping $u$ from a topological space $(\mathfrak{Y},\mathfrak{T})$ into $(\mathfrak{Z}',\sigma(\mathfrak{Z}',\mathfrak{Z}))$ is continuous if and only if for each $z \in \mathfrak{Z}$ the map $\langle u(\cdot),z\rangle$ is continuous on $\mathfrak{Y}$ into $\mathbb
    F$ (compare with \cite[Page 51]{Schaefer}).
    \item If $(\mathfrak{Y},\rho)$ is a metric space, then the continuity of $u:\mathfrak{Y}\to \mathfrak{G}$, where $(\mathfrak{G},\mathfrak{S})$ is an arbitrary topological space, is equivalent to its sequential continuity.
\end{enumerate}

If we are given a linear map $u:(\mathfrak{Z}',\sigma(\mathfrak{Z}',\mathfrak{Z}))\to (\mathfrak{F}',\sigma(\mathfrak{F}',\mathfrak{F}))$, where $\mathfrak{F}$ is another Fr\'echet space, then in virtue of \cite[Proposition 2.1, page 128]{Schaefer} we have that $u$ is continuous if and only if $u^\ast(\mathfrak{F}) \subset \mathfrak{Z}$, where $u^\ast: (\mathfrak{F}')^\ast\to (\mathfrak{Z}')^\ast$ is the algebraic adjoint of $u$ and
\begin{equation*}
    \langle uz', f \rangle = \langle z',u^\ast f \rangle, \quad z'\in \mathfrak{Z}',\, f \in \mathfrak{F}\subset (\mathfrak{F}')^\ast\,.
\end{equation*}

A \texttt{barrel} in a topological vector space $\mathfrak{X}$ is a subset which is radial, convex, circled and closed. A locally convex space $\mathfrak{X}$ is \texttt{barelled} if each barrel in $\mathfrak{X}$ is a neighborhood of 0. In virtue of \cite[Corollary on page 60]{Schaefer}), every Fr\'echet space is barelled. 

A topological vector space $\mathfrak{X}$ is a \texttt{Montel} space if it is barelled and if every closed and bounded subset of $\mathfrak{X}$ is compact. We remark that, if $\mathfrak{X}$ is assumed to be normed, then it must necessarily be finite-dimensional.

Let us examine the consequences of the special case in which $\mathfrak{Z}$, in addition to be a Fr\'echet space, is assumed to be a Montel space. Theorem 5.9 at page 147 in \cite{Schaefer} guarantees that $(\mathfrak{Z}',\beta(\mathfrak{Z}',\mathfrak{Z}))$ is itself a Montel space. Furthermore, in view of the following two general facts \cite[page 194]{Schaefer}
\begin{enumerate}
    \item A separable Fr\'echet space $\mathfrak{E}$ is a Montel space if and only if each  $\sigma(\mathfrak{E}',\mathfrak{E})$-convergent sequence in $\mathfrak{E}'$ is strongly convergent,
    \item A metrisable Montel space is separable,
\end{enumerate}  
we have the equivalence of weak and strong convergence of sequences in $\mathfrak{Z}'$, namely $z'_n \to 0$ wrt $\sigma(\mathfrak{Z}',\mathfrak{Z})$ if and only if $z'_n\to 0$ wrt $\beta(\mathfrak{Z}',\mathfrak{Z})$.

Following \cite[page 202]{Schaefer}, by an \texttt{algebra}, we mean an $\mathbb F$-vector space $A$ endowed with a bilinear map $A\times A \to A$ called \texttt{multiplication} and usually denoted by $(a,b)\mapsto a\cdot b$. This product is neither assumed associative nor carrying a unit. An algebra $A$ will be termed a \texttt{locally convex algebra} if the underlying space is a locally convex space and if the multiplication is separately continuous. 

\subsection{The Hida spaces}

The Hida spaces serve as the main example of Fr\'echet spaces in our exposition. For the convenience of the reader, we make a brief introduction of the basics of these spaces of smooth and generalized random variables. Our presentation is based on the classical book of Hida {\it et al.} \cite{HKPS}, where further details and theory can be found. In the way we naturally encounter the Schwartz distribution spaces, another example of a Fr\'echet space.

With $n\in\mathbb N$, we denote by $\mathcal S(\mathbb R^n)$ the Schwartz functions on $\mathbb R^n$, with topological dual $\mathcal S'(\mathbb R^n)$. For the second order differential (Hamiltonian) operator $A:=-d^2/dx^2+(1+x^2)$ and  $\vert\cdot\vert_{2,0}:=\vert\cdot\vert_2$ being the norm in $L^2(\mathbb R^n)$, introduce $\vert\cdot\vert_{2,p}:=\vert(A^{\otimes n})^p\cdot\vert_2$ for $p\in\mathbb N\cup\{0\}$. With the Hilbert spaces $\mathcal S_p(\mathbb R^n)$ being the completion of $\mathcal S(\mathbb R^n)$ in the norm $\vert\cdot\vert_{2,p}$, one has that $\mathcal S(\mathbb R^n)$ is the projective limit of $\mathcal S_p(\mathbb R^n)$, while the tempered Schwartz distributions $\mathcal S'(\mathbb R^n)$ is the inductive limit of the dual spaces, denoted $\mathcal S_{-p}(\mathbb R^n)$. 

Define a probability space $(\mathcal S'(\mathbb R), \mathcal F,\mathbb P)$ where $\mathcal F$ is the $\sigma$-algebra induced by the weak topology and the probability measure $\mathbb P$ is defined by the Bochner-Minlos theorem as
$$
\int_{\mathcal S'(\mathbb R)}\exp((\omega,f))\mathbb P(d\omega)=\exp(-\frac12\vert f\vert_2^2)
$$
where $(\omega,f)$ is the dual pairing between $\mathcal S'(\mathbb R)$ and $\mathcal S(\mathbb R)$. We see that 
$\omega\mapsto W(f)(\omega):=(\omega,f)$ defines a mean-zero Gaussian random variable with variance $\vert f\vert_2^2$. In fact, Brownian motion $B(t):=(\cdot,\mathrm{1}_{[0,t]})$ exists as a limit in $(L^2):=L^2(\mathcal S'(\mathbb R),\mathcal F,\mathbb P)$. Any random variable $X\in(L^2)$ has a so-called chaos expansion
$$
X=\sum_{n=0}^{\infty}I_n(f^{(n)}),
$$
where $I_n$ are the $n$-fold Wiener-It\^o integrals on $\mathbb R^n$ and $f^{(n)}\in L^2(\mathbb R^n)$ are symmetric functions. The $(L^2)$-norm of $X$, being its second moment, is
$$
\Vert X\Vert_{2}^2=\sum_{n=0}^{\infty}n!\vert f^{(n)}\vert_2^2.
$$
In passing, we introduce the Malliavin derivative (see e.g. Nualart \cite{Nualart-book}) of elements in a subspace of $(L^2)$. Denote by $\mathcal D_{1,2}$ the subspace of $(L^2)$ for which 
$$
\sum_{n=0}^{\infty}n n!\vert f^{(n)}\vert_2^2<\infty.
$$
The Malliavin derivative $D_t$ for $t\in\mathbb R$ of $X\in\mathcal D_{1,2}$ is defined as the element $D_tX\in(L^2)$ with chaos expansion
$$
D_tX:=\sum_{n=1}^{\infty}n I_{n-1}(f^{(n)}(\cdot,t)).
$$
The Malliavin derivative is central in the definition of the Skorohod stochastic integral.  

To define the Hida test functions (or, smooth random variables), denote by $(\mathcal S)_p$ the Hilbert space of random variables $\phi\in(L^2)$ with chaos functions $f^{(n)}\in\mathcal S_p(\mathbb R^n)$ for which
$$
\Vert \phi\Vert_{2,p}^2:=\sum_{n=0}^{\infty}n!\vert f^{(n)}\vert_{2,p}^2<\infty.
$$
Letting $(\mathcal S)_{-p}$ denote the dual, we define $(\mathcal S)$ as the projective limit of $(\mathcal S)_{p}$ and $(\mathcal S)'$ the inductive limit of $(\mathcal S)_{-p}$. The spaces $(\mathcal S)$ and $(\mathcal S)'$ are known as the Hida test function and distribution space, resp. Formally, elements $\Phi\in(\mathcal S)'$ has chaos expansion
$$
\Phi=\sum_{n=0}^{\infty}I_n(F^{(n)})
$$
with symmetric elements $F^{(n)}\in\mathcal S'(\mathbb R^n)$, where the duality pairing between
$(\mathcal S)$ and $(\mathcal S)'$ becomes
$$
\langle\Phi,\phi\rangle=\sum_{n=0}^{\infty}n!(F^{(n)},f^{(n)}).
$$
The canonical example of a Hida distribution is {\it white noise}, being the generalized random variable $W(t):=I_1(\delta_t)$. Here, $\delta_t\in\mathcal S'(\mathbb R)$ is the Dirac-$\delta$ function. 

The $\mathcal S$-transform $\mathcal S\Phi$ of an element $\Phi\in(\mathcal S)'$ is
$$
\mathcal S(\mathbb R)\ni\xi\mapsto \mathcal S\Phi(\xi):=\langle\Phi,\exp(W(\xi)-\frac12\vert\xi\vert_2^2)\rangle 
$$
which is a bijection between $(\mathcal S)'$ and the space of so-called $\mathcal U$-functionals. The $\mathcal S$-transform provides a convenient characterisation of Hida distributions, as well as being a powerful computational tool for analysis. The Wick product between two elements $\Phi,\Psi\in(\mathcal S)'$ , denoted $\Phi\diamond\Psi$, is defined as a "convolution product" for the $\mathcal S$-transform, 
$$
\Phi\diamond\Psi:=\mathcal S^{-1}(\mathcal S\Phi\cdot\mathcal S\Psi).
$$
We remark in passing that in the literature, the Hida distribution space is frequently denoted by $(\mathcal S)^*$, however, we reserve that notation for the algebraic dual.  

Notice that both $\mathcal S(\mathbb R^n)$ and $(\mathcal S)$, the Schwartz and Hida test function spaces, are nuclear Frechet spaces. Moreover, they are both also Montel spaces (see Hida {\it et al. \cite[page 483]{HKPS})}. In the case of the Hida space $(\mathcal S)$, recall that the bounded subsets of $\mathfrak{Z}$ are precisely the ones which are bounded in every norm $\abs{\cdot}_p,\, p\in \N_0$ (see Hida {\it et al.} \cite[page 481]{HKPS}).

We end this section with noticing that there exists many other triplets of spaces $(\mathcal D)\subset(L^2)\subset(\mathcal D)'$ based on chaos expansions and projective and inductive limits. For example, Holden {\it et al.} \cite{HOUZ} introduces the Kondratiev spaces of smooth and generalized random variables.

\section{The Fr\'echet-valued Gelfand integral}
\label{sec:Gelfand}
In this section we introduce the Fr\'echet-valued Gelfand integral and study in detail several properties of it.  

To this end, let $(M,\mathcal{M},\mu)$ be an arbitrary $\sigma$-finite measure space, and let us consider an $\mathbb{F}$-Fr\'echet space $\mathfrak{Z}$. Let $\mathfrak{Z}'$ designate its topological dual. Let us consider a map

\[
\psi: M \to \mathfrak{Z}'
\]
such that for any $z\in\mathfrak{Z}$, the function $x\mapsto \langle\psi(x),z \rangle$ is an element of $L^1(\mu)$. Here we reserve the notation $\langle\cdot,\cdot\rangle$ for the dual pairing between elements in $\mathfrak{Z}'$ and $\mathfrak{Z}$. Furthermore, we will denote the norm in $L^1(\mu)$ by $\Vert\cdot\Vert_{L^1(\mu)}$.

The following result is an easy generalization of Hida {\it et al.} \cite[Prop. 8.1]{HKPS} (where the authors again refer to Hille and Phillips \cite{HS} for the proof in the Banach case):
\begin{proposition}\label{prop: the integral}
Let $E\in \mathcal{M}$. Then the map $\Lambda_E:\mathfrak{Z}\to \mathbb{F}$ given by
\[
\Lambda_E(z) = \int_E \langle \psi(x),z\rangle\,\mu(dx)\,,\;\; z\in \mathfrak{Z}
\]
belongs to $\mathfrak{Z}'$.
\end{proposition}
\begin{proof}
We adapt the proof of Hida {\it et al.} \cite[Prop. 8.1]{HKPS}. $\Lambda_E$ is clearly well defined and linear, thus only continuity needs to be checked. We set 
\begin{equation}
\label{eq:Psi-op}
\Psi_E:\mathfrak{Z}\to L^1(\mu)\,,\;\;\; z\mapsto\langle I_E(\cdot)\psi(\cdot),z\rangle\,,
\end{equation}
where $I_E$ denotes the indicator function on $E$. We claim that $\Psi_E$ is closed. Indeed, given $z_n\overset{\mathfrak{Z}}{\to}z$ and $\Psi_E(z_n)\overset{L^1(\mu)}{\longrightarrow}v$, as $n\to \infty$, we have clearly $\langle I_E(x)\psi(x),z_n-z\rangle\to 0$, for any $x\in M$. 

On the other hand, by reverse dominated convergence theorem, there exists a subsequence $(n_j)_j$ such that $\Psi_E(z_{n_j})\to v$ $\mu$-a.e. $x\in M$, as $j\to\infty$. Therefore, $\Psi_E(z_{n_j})\to \Psi_E(z)$ $\mu$-a.e. $x\in M$, implying $v=\Psi_E(z)$ almost everywhere. Hence $\Psi_E$ is closed. By the closed graph theorem (see \cite[Thm. 2.3 on page 78]{Schaefer}), this is equivalent to continuity of $\Psi_E$ from $\mathfrak{Z}$ into $L^1(\mu)$. Since
$$
\Lambda_E(z)=\int_M\langle I_E(x)\psi(x),z\rangle \,\mu(dx) = \int_M \Psi_E(z)\,\mu(dx)\,
$$
and $\int_M\cdot \,\mu(dx)$ is linear and continuous from $L^1(\mu)$ into $\mathbb{F}$, we conclude that $\Lambda_E\in \mathfrak{Z}'$.
\end{proof}

In view of Proposition \ref{prop: the integral}, we give the following definition:
\begin{definition}
A map $\psi: M \to \mathfrak{Z}'$ such that for any $z\in\mathfrak{Z}$ the function $\langle\psi(\cdot),z \rangle$ belongs to $L^1(\mu)$ will be termed $\mathfrak{Z}'$-Gelfand $\mu$-integrable.
\end{definition}

From now on, for such a $\psi$ we set $\int_E\psi(x)\,\mu(dx):=\Lambda_E\in \mathfrak{Z}'$, for any $E\in\mathcal{M}$, and note that by definition it holds
$$
\int_E\langle \psi(x),z\rangle \,\mu(dx) = \left\langle \int_E   \psi(x) \,\mu(dx),z  \right\rangle \,.
$$
Let us next show some useful results and properties of this abstract integral.

Recall that we may assume that there exists a non decreasing sequence of seminorms $(p_n)_{n\in \N}$ which generates the topology of $\mathfrak{Z}$. Since for any $E\in \mathcal{M}$ the linear operator $\Psi_E$ defined in \eqref{eq:Psi-op} of the proof of Proposition \ref{prop: the integral} is continuous from $\mathfrak{Z}$ into $L^1(\mu)$, we infer (see \cite[Prop. 1.1. on page 74]{Schaefer}) that there exists a constant $C(E)\geq 0$ and an integer $n(E)\in \N$ such that
$$
\norm{\Psi_E(z)}_{L^1(\mu)}\leq C(E)\, p_{n(E)}(z)\,, \;\; z\in\mathfrak{Z}\,.
$$
Observe that if $\tilde{E}\in \mathcal{M}$ with $\tilde{E}\subset E$, then it holds $\norm{\Psi_{\tilde{E}}(z)}_{L^1(\mu)}\leq C(E)\, p_{n(E)}(z)$ for all $z\in \mathfrak{Z}$, since trivially $\norm{\Psi_{\tilde{E}}(z)}_{L^1(\mu)}\leq \norm{\Psi_E(z)}_{L^1(\mu)}$ for all $z$. Therefore, we obtain the following universal bound
$$
\norm{\Psi_E(z)}_{L^1(\mu)}\leq C(M)\, p_{n(M)}(z)\,,
$$
valid for all $E\in \mathcal{M}$ and $z\in\mathfrak{Z}$.

Observe also that, for any $E,\tilde{E}\in \mathcal{M}$, trivially one has $\abs{I_E(x) - I_{\tilde{E}}(x)}= I_{E\Delta \tilde{E}}(x)$, where $\Delta$ stands for the symmetric difference of two subsets. In view of this, the following fundamental inequality holds
\[
\begin{split}
\abs{\Lambda_E(z) - \Lambda_{\tilde{E}}(z)}  &= \abs{\int_M \langle (I_E(x) - I_{\tilde{E}}(x))\psi(x),z\rangle\,\mu(dx)  } \\
& \leq \int_M \abs{I_E(x) - I_{\tilde{E}}(x)}\abs{\langle \psi(x),z\rangle}\,\mu(dx)  \\
& =  \int_M \abs{\langle I_{E\Delta \tilde{E}}(x)\,\psi(x),z\rangle}\,\mu(dx) \\
& = \norm{\Psi_{E\Delta\tilde{E}}(z)}_{L^1(\mu)} \\
&\leq C(M)\, p_{n(M)}(z) 
\end{split}
\]
valid for any $z\in \mathfrak{Z}$.

We define $\Gamma_{E,\tilde{E}}:= \Lambda_E - \Lambda_{\tilde{E}}\in \mathfrak{Z}'$, with $E,\tilde{E}\in \mathcal{M}$. Then, from the last inequality we infer
\[
\abs{\Gamma_{E,\tilde{E}}(z)} \leq \norm{\Psi_{E\Delta\tilde{E}}(z)}_{L^1(\mu)} 
\leq C(M)\, p_{n(M)}(z)\,,
\]
for all $z\in\mathfrak{Z}$ and $E,\tilde{E}\in \mathcal{M}$. It hence follows that $\{\Gamma_{E,\tilde{E}}\}_{E,\tilde{E}}$ is a simply bounded family of $\mathfrak{Z}'$ (see \cite[page 82]{Schaefer}). Since $\mathfrak{Z}$ is barelled, by means of \cite[Thm. 4.2 on page 83]{Schaefer} we get that $\{\Gamma_{E,\tilde{E}}\}_{E,\tilde{E}}$ is equicontinuous. By the Alaoglu-Bourbaki theorem (see \cite[page 84]{Schaefer}), the family is relatively compact with respect to $\sigma(\mathfrak{Z}',\mathfrak{Z})$. In conclusion, we have proved the following result:

\begin{lemma}\label{lemma: fundamental inequality}
Let $E,\tilde{E}\in \mathcal{M}$. Then for all $z\in \mathfrak{Z}$ it holds
\begin{equation}\label{eq: fundamental inequality}
\abs{\Lambda_E(z) - \Lambda_{\tilde{E}}(z)}  \leq
\int_{E\Delta \tilde{E}} \abs{\langle \psi(x),z\rangle}\,\mu(dx)\,.
\end{equation}
Furthermore, the family $\{\Gamma_{E,\tilde{E}}\}_{E,\tilde{E} \in \mathcal{M}}\subset \mathfrak{Z}'$, where $\Gamma_{E,\tilde{E}}:= \Lambda_E - \Lambda_{\tilde{E}}$, is relatively compact with respect to $\sigma(\mathfrak{Z}',\mathfrak{Z})$.
\end{lemma}



To show ``continuity'' of $\Lambda_E$ as a function of $E$, we proceed as follows. First, introduce the Fr\'echet-Nikodym metric on $\mathcal{M}/{\sim}$, where $\sim$ denotes the equivalence relation on $\mathcal{M}$ defined as $E\sim\tilde{E}$ whenever $\mu(E\Delta\tilde{E})=0$. The Fr\'echet-Nikodym metric $d_{\mu}$ is then  
$$
d_{\mu}: (\mathcal{M}/{\sim})^2\to [0,\infty]\,,\;\;\; ([E],[\tilde{E}])\mapsto d_{\mu}([E],[\tilde{E}]):= \mu(E\Delta\tilde{E})\,.
$$
Observe that the metric is allowed to assume the value $\infty$, because the measure space is assumed to be $\sigma$-finite only. 

Notice also that if $E\sim\tilde{E}$, then clearly for all $z\in \mathfrak{Z}$ we have
\[
\begin{split}
\int_E\langle \psi(x),z\rangle \,\mu(dx) & = \int_{E\setminus\tilde{E}}\langle \psi(x),z\rangle \,\mu(dx) + \int_{E\cap\tilde{E}}\langle \psi(x),z\rangle \,\mu(dx)      \\
&= \int_{E\cap\tilde{E}}\langle \psi(x),z\rangle \,\mu(dx)  +
\int_{\tilde{E}\setminus E}\langle \psi(x),z\rangle \,\mu(dx) \\
&=\int_{\tilde{E}}\langle \psi(x),z\rangle \,\mu(dx)\,, 
\end{split}
\]
namely, we can unambiguously define a map from $\mathcal{M}/{\sim}$ into $\mathfrak{Z}'$ as
$$
[E]\mapsto\int_E\psi(x)\,\mu(dx)
$$

We are ready to prove
\begin{proposition}
\label{prop: weak star continuity wrt FN metric}
The map $[E]\mapsto\int_E\psi(x)\,\mu(dx)$ is continuous from $(\mathcal{M}/{\sim},d_{\mu})$ into $(\mathfrak{Z}',\sigma(\mathfrak{Z}',\mathfrak{Z}))$, i.e. weak star continuous. 
\end{proposition}
\begin{proof}
Fix $z\in\mathfrak{Z}$ and $\varepsilon>0$. By assumption $\langle\psi(\cdot),z \rangle$ is in $L^1(\mu)$. By the absolute continuity of the Lebesgue integral (see e.g. Billingsley \cite{B}), 
it follows that there exists a $\delta>0$ such that for any $F\in\mathcal{M}$ with $\mu(F)<\delta$, then $\int_F  \abs{\langle\psi(x),z \rangle}\,\mu(dx)<\varepsilon$. 

Therefore, for any $[E],[\tilde{E}]\in \mathcal{M}/{\sim}$ such that $d_{\mu}([E],[\tilde{E}])<\delta$, we will have $\mu(E\Delta\tilde{E})<\delta$, and thus inequality \eqref{eq: fundamental inequality} in Lemma \ref{lemma: fundamental inequality} implies
\[
\begin{split}
\abs{
\left\langle \int_E   \psi(x) \,\mu(dx),z  \right\rangle - 
\left\langle \int_{\tilde{E}}   \psi(x) \,\mu(dx),z  \right\rangle
}
& \leq  \int_{E\Delta \tilde{E}}
\abs{\langle\psi(x),z\rangle}\,\mu(dx) < \varepsilon
\end{split}
\]
provides us with the desired result.
\end{proof}
Observe that we are not requiring any improved integrability assumptions on $\psi$ in the Proposition \ref{prop: weak star continuity wrt FN metric} above.

For any $y=(y_1,\dots,y_k)\in\R^k$, let $E_y$ denote the subset $\times_{i=1}^k(-\infty,y_i]$. As an easy corollary (but with important consequences) we have
\begin{corollary}
\label{cor:integral-limit-continuity}
Let  $(M,\mathcal{M},\mu)=(\R^k,\mathcal{B}^k,Leb)$, where $Leb$ is the $k$-dimensional Lebesgue measure and $\mathcal{B}^k$ the Borel $\sigma$-algebra on $\mathbb R^k$. Then the map
$$
y\mapsto\int_{E_y}\psi(x)\,dx
$$
is continuous from $\R^k$ into $(\mathfrak{Z}',\sigma(\mathfrak{Z}',\mathfrak{Z}))$. 
\end{corollary}
\begin{proof}
Fix $y\in \R^k$ and let $(y_j)_j\subset\R^k$ be an arbitrary sequence converging to $y$. Then,  
it holds
$$
d_{Leb}([E_y],[E_{y_j}])= Leb(E_y\Delta E_{y_j})\to 0
$$
as $j\to \infty$, and hence the result follows.
\end{proof}

Finally, let us treat the special case in which $\mathfrak{Z}$, in addition to be a Fr\'echet space, is assumed to be a Montel space. In view of the facts stated in Section \ref{sec: rapid excursion}, the previous results may upgraded to 
\begin{proposition}\label{prop: strong continuity wrt FN metric}
Assume additionally that $\mathfrak{Z}$ is Montel. Then the map $[E]\mapsto\int_E\psi(x)\,\mu(dx)$ is continuous from $(\mathcal{M}/{\sim},d_{\mu})$ into
$(\mathfrak{Z}',\beta(\mathfrak{Z}',\mathfrak{Z}))$. 
\end{proposition}
\begin{proof}
Fix $[E]\in \mathcal{M}/{\sim}$ and let $([E_n])_n\subset \mathcal{M}/{\sim}$ be an arbitrary sequence converging to $[E]$ with respect to $d_\mu$. Then, Proposition \ref{prop: weak star continuity wrt FN metric} guarantees that
\begin{equation*}
    \int_{E_n} \psi(x)\,\mu(dx) \to  \int_{E} \psi(x)\,\mu(dx)\,,\quad \text{with respect to } 
    \sigma(\mathfrak{Z}',\mathfrak{Z})
\end{equation*}
and hence, in view of the discussion in Section \ref{sec: rapid excursion}, with respect to $\beta(\mathfrak{Z}',\mathfrak{Z}))$ as well, namely we have sequential continuity. This is sufficient to conclude, because $(\mathcal{M}/{\sim},d_{\mu})$ is a metric space.
\end{proof}
Analogously, we have the specific case,
\begin{corollary}\label{corr: strong continuity wrt FN metric}
Let  $(M,\mathcal{M},\mu)=(\R^k,\mathcal{B}^k,Leb)$, where $Leb$ is the $k$-dimensional Lebesgue measure and $\mathcal{B}^k$ the Borel $\sigma$-algebra on $\mathbb R^k$. Assume additionally that $\mathfrak{Z}$ is Montel. Then the map
$$
y\mapsto\int_{E_y}\psi(x)\,dx
$$
is continuous from $\R^k$ into $(\mathfrak{Z}',\beta(\mathfrak{Z}',\mathfrak{Z}))$. 
\end{corollary}

\subsection{Properties of the Gelfand integral}

In this subsection we show a collection of some important results on the Gelfand integral, including a Fubini theorem, convergence theorems and linear transformations on the integral.  

First, let us show a Fubini-Tonelli type theorem for the integral.

\begin{theorem}
\label{thm:fubini}
Let $\mathfrak{Z}$ be an $\mathbb{F}$-Fr\'echet space and $(M_i,\mathcal{M}_i,\mu_i),i=1,2$ be two arbitrary $\sigma$-finite measure spaces. Set $(M,\mathcal{M},\mu):=(M_1\times M_2,\mathcal{M}_1\otimes\mathcal{M}_2 ,\mu_1\otimes \mu_2)$. Consider a mapping  
\[
\psi: M \to \mathfrak{Z}'
\]
that is $\mathfrak{Z}'$-Gelfand $\mu$-integrable. 

Then,
\begin{equation*}
    \begin{split}
    \int_{M_1}\int_{M_2} \langle\psi(x_1,x_2),\cdot\rangle \mu_2(dx_2) \mu_1(dx_1)  & = \int_M   \psi(x) \,\mu(dx)  \\
    &= \int_{M_2}\int_{M_1} \langle\psi(x_1,x_2),\cdot\rangle \mu_1(dx_1) \mu_2(dx_2)
    \end{split}
\end{equation*}
as elements of $\mathfrak{Z}'$.
\end{theorem}

\begin{proof}
Fix once for all $z\in\mathfrak{Z}$. By the ``classical'' Fubini-Tonelli Theorem there exists $F_1(z)\in \mathcal{M}_1$ such that $\mu_1(M_1\setminus F_1(z))=0$, for any $x_1\in F_1(z)$ 
\begin{equation*}
    \langle\psi(x_1,\cdot),z\rangle \in L^1(\mu_2),
\end{equation*}
and the $\mu_1$-a.e. defined function 
\begin{equation*}
    x_1\mapsto \int_{M_2} \langle\psi(x_1,x_2),z\rangle \mu_2(dx_2) 
\end{equation*}
is an element of $L^1(\mu_1)$. 

Thus, again in force of the "classical" Fubini-Tonelli Theorem, and directly from the definition of the Gelfand integral,
\begin{equation*}
    \begin{split}
    \int_{M_1}\int_{M_2} \langle\psi(x_1,x_2),z\rangle \mu_2(dx_2) \mu_1(dx_1) &=
    \int_{M} \langle\psi(x),z\rangle \mu(dx)\\
    & =\left\langle \int_M   \psi(x) \,\mu(dx),z  \right\rangle.
    \end{split}
\end{equation*}
By "symmetry", we must also have
\begin{equation*}
    \begin{split}
    \int_{M_2}\int_{M_1} \langle\psi(x_1,x_2),z\rangle \mu_1(dx_1) \mu_2(dx_2) &=
    \int_{M} \langle\psi(x),z\rangle \mu(dx)\\
    & =\left\langle \int_M   \psi(x) \,\mu(dx),z  \right\rangle .
    \end{split}
\end{equation*}
This holds for any $z\in\mathfrak{Z}$. But $\int_M   \psi(x) \,\mu(dx)$ is an element of $\mathfrak{Z}'$. It follows that also the other terms belong to $\mathfrak{Z}'$, and thus we conclude (equality in $\mathfrak{Z}'$)
\begin{equation*}
    \begin{split}
    \int_{M_1}\int_{M_2} \langle\psi(x_1,x_2),\cdot\rangle \mu_2(dx_2) \mu_1(dx_1)  & = \int_M   \psi(x) \,\mu(dx)  \\
    &= \int_{M_2}\int_{M_1} \langle\psi(x_1,x_2),\cdot\rangle \mu_1(dx_1) \mu_2(dx_2).
    \end{split}
\end{equation*}
The Theorem is proved.
\end{proof}

In view of this last result, it makes sense to define ($z\in\mathfrak{Z}$)
\begin{equation*}
    \left\langle \int_{M_1}\int_{M_2} \psi(x_1,x_2)\, \mu_2(dx_2) \mu_1(dx_1),z\right\rangle := \int_{M_1}\int_{M_2} \langle\psi(x_1,x_2),z\rangle \mu_2(dx_2) \mu_1(dx_1)
\end{equation*}
and 
\begin{equation*}
    \left\langle \int_{M_2}\int_{M_1} \psi(x_1,x_2)\, \mu_1(dx_1) \mu_2(dx_2),z\right\rangle := \int_{M_2}\int_{M_1} \langle\psi(x_1,x_2),z\rangle \mu_1(dx_1) \mu_2(dx_2).
\end{equation*}
Therefore, the previous result may be re-stated in a more compact and familiar way as 
\begin{equation*}
    \begin{split}
     \int_{M_1}\int_{M_2} \psi(x_1,x_2) \mu_2(dx_2) \mu_1(dx_1)  & = \int_M   \psi(x) \,\mu(dx)  \\
    &= \int_{M_2}\int_{M_1} \psi(x_1,x_2) \mu_1(dx_1) \mu_2(dx_2),   
    \end{split}
\end{equation*}
as elements of $\mathfrak{Z}'$.

\begin{remark}
In the proof of Thm. \ref{thm:fubini} it is worth noticing the remarkable feature that the $z$-dependency subsets $F_1(z)$ does not become a matter of concern in the end. By the interpretation of the double integrals as the $\mathfrak Z'$-elements 
$$
z\mapsto\int_{M_1}\int_{M_2}\langle\psi(x_1,x_2),z\rangle\mu_2(dx_2)\mu_1(dx_1)
$$
(and similarly for the commuted integral), we see that we cannot "move" the $z$ through the two integrals iteratively, as we might hit outside $F_1(z)$. 
\end{remark}

We next move our attention to convergence theorems for the integral. First we show a Vitali convergence theorem.

\begin{theorem}\label{thm: Vitali}
Let $(M,\mathcal{M},\mu)$ be an arbitrary $\sigma$-finite measure space, and let $\mathfrak{Z}$ be an $\mathbb{F}$-Fr\'echet space. Consider a sequence of functions $\{\psi_k: M \to \mathfrak{Z}'\}_{k\in \N}$  and  $\psi: M \to \mathfrak{Z}'$
such that
\begin{enumerate}
    \item[(i)] for any $k\in \N$ the function $\psi_k$ is $\mathfrak{Z}'$-Gelfand $\mu$-integrable;
    \item[(ii)] for any $z\in \mathfrak{Z}$ the function $x\mapsto \langle\psi(x),z \rangle$ is $\mu$-measurable; 
    \item[(iii)] for $\mu$-a.e. $x\in M$, $\psi_k(x)\to \psi(x)$ in $(\mathfrak{Z}',\sigma(\mathfrak{Z}',\mathfrak{Z}))$ as $k\to \infty$;
    \item[(iv)] for any $z\in \mathfrak{Z}$ and $\varepsilon>0$ there exists $E_{\varepsilon,z}\in \mathcal{M}$ with $\mu(E_{\varepsilon,z})<\infty$ such that for all $G\in\mathcal{M}$ with $G\cap E_{\varepsilon,z}=\emptyset$ it holds
    \begin{equation}\label{eq: Vitali 1}
      \sup_k \int_G \abs{\langle\psi_k(x),z\rangle} \mu(dx) < \varepsilon\,;
    \end{equation}
    \item[(v)] for any $z\in \mathfrak{Z}$ and $\varepsilon>0$ there exists $\delta_{\varepsilon,z}>0$ such that for all $E\in\mathcal{M}$ with $\mu(E)<\delta_{\varepsilon,z}$ it holds
    \begin{equation}\label{eq: Vitali 2}
      \sup_k \int_E \abs{\langle\psi_k(x),z\rangle} \mu(dx) < \varepsilon\,.
    \end{equation}
\end{enumerate}
Then $\psi$ is $\mathfrak{Z}'$-Gelfand $\mu$-integrable, and as $k\to \infty$
\begin{equation*}
    \int_M \psi_k(x) \mu(dx) \to \int_M \psi(x) \mu(dx) \quad \text{in } (\mathfrak{Z}',\sigma(\mathfrak{Z}',\mathfrak{Z}))\,.
\end{equation*}
\end{theorem}

\begin{proof}
Fix $z\in \mathfrak{Z}$. By Fatou's Lemma (see Folland \cite{Folland}),
\begin{equation*}
    \int_M \abs{\langle\psi(x),z\rangle} \mu(dx) \leq \liminf_k  \int_M \abs{\langle\psi_k(x),z\rangle} \mu(dx) \leq \infty\,.
\end{equation*}
Let $\varepsilon = 1$. Then, according to \eqref{eq: Vitali 1}, there exists $E_{1,z}\in \mathcal{M}$ with $\mu(E_{1,z})<\infty$ such that
\begin{equation}\label{eq: Vitali 3}
    \sup_k \int_{M\setminus E_{1,z}} \abs{\langle\psi_k(x),z\rangle} \mu(dx) < 1\,.
\end{equation}
Let $\delta_{1,z}>0$ be given by \eqref{eq: Vitali 2}. Then, by the Severini-Egoroff Theorem (see Folland \cite{Folland}) there exists a measurable subset $R_{1,z}\subset E_{1,z}$ such that $\mu(R_{1,z})<\delta_{1,z}$ and
\begin{equation}\label{eq: Vitali 4}
    \langle \psi_k(x),z\rangle \to \langle \psi(x),z\rangle
\end{equation}
uniformly on $E_{1,z}\setminus R_{1,z}$. Furthermore, by \eqref{eq: Vitali 2}, we have
\begin{equation*}
    \sup_k \int_{R_{1,z}} \abs{\langle\psi_k(x),z\rangle} \mu(dx) < 1\,.
\end{equation*}
This fact joined with \eqref{eq: Vitali 3} and \eqref{eq: Vitali 4} leads us to, for some suitable $k_z\in \N$ and all $k>k_z$ 
\begin{equation*}
    \begin{split}
       \int_M \abs{\langle\psi_k(x),z\rangle} \mu(dx)
       &\leq\int_{M\setminus E_{1,z}}\abs{\langle\psi_k(x),z\rangle} \mu(dx)+\int_{E_{1,z}\setminus R_{1,z}}\abs{\langle\psi_k(x),z\rangle} \mu(dx) \\
       &\qquad+\int_{R_{1,z}}\abs{\langle\psi_k(x),z\rangle} \mu(dx)\\ 
       & <  2  + \int_{E_{1,z}\setminus R_{1,z}} \abs{\langle\psi_k(x),z\rangle} \mu(dx) \\
       & <  2  + \int_{E_{1,z}\setminus R_{1,z}} \abs{\langle\psi_k(x) - \psi_{k_z}(x),z\rangle} \mu(dx) \\
       & \qquad+ \int_{E_{1,z}\setminus R_{1,z}} \abs{\langle\psi_{k_z}(x),z\rangle} \mu(dx)\\
       & <  3 + \int_{E_{1,z}\setminus R_{1,z}} \abs{\langle\psi_{k_z}(x),z\rangle} \mu(dx) \\
       & <  3 + \int_M \abs{\langle\psi_{k_z}(x),z\rangle} \mu(dx)\\
       & < \infty\,.       
    \end{split}
\end{equation*}
Therefore, $\liminf_k  \int_M \abs{\langle\psi_k(x),z\rangle} \mu(dx) < \infty$ and hence $\langle\psi(\cdot),z \rangle$ is an element of $L^1(\mu)$. 

In view of this, it is evident that the ``enlarged'' sequence $\{\langle\psi_k(\cdot),z\rangle\}_{k\in\N}\cup \{\langle\psi(\cdot),z\rangle\}$ still satisfies properties \eqref{eq: Vitali 1} and \eqref{eq: Vitali 2} with possibly different $E_{\varepsilon,z}$ and $\delta_{\varepsilon,z}$. Therefore, for $\varepsilon>0$ and arguing as we have done above, it holds
\begin{equation*}
    \begin{split}
       \int_M \abs{\langle\psi_k(x),z\rangle - \langle\psi(x),z\rangle} \mu(dx) 
       & =  
       \int_{M\setminus E_{\varepsilon,z}} \abs{\langle\psi_k(x),z\rangle - \langle\psi(x),z\rangle} \mu(dx) \\
       & \qquad + \int_{ E_{\varepsilon,z}} \abs{\langle\psi_k(x),z\rangle - \langle\psi(x),z\rangle} \mu(dx) \\
       & <  2\varepsilon  + \int_{R_{\varepsilon,z}} \abs{\langle\psi_k(x) - \psi(x),z\rangle} \mu(dx) \\
       & \qquad + \int_{E_{\varepsilon,z}\setminus R_{\varepsilon,z}} \abs{\langle\psi_k(x) - \psi(x),z\rangle} \mu(dx)\\
       & <  4\varepsilon + \int_{E_{\varepsilon,z}\setminus R_{\varepsilon,z}} \abs{\langle\psi_k(x) - \psi(x),z\rangle} \mu(dx)\,.       
    \end{split}
\end{equation*}
By uniform convergence, sending first $k\to\infty$ and then $\varepsilon\to 0$, we obtain that $\lim_k \int_M \abs{\langle\psi_k(x),z\rangle - \langle\psi(x),z\rangle} \mu(dx)=0$ for arbitrary $z\in \mathfrak{Z}$. By Proposition \ref{prop: the integral}, we conclude  
\begin{equation*}
    \int_M \psi_k(x) \mu(dx) \to \int_M \psi(x) \mu(dx) \quad \text{in } (\mathfrak{Z}',\sigma(\mathfrak{Z}',\mathfrak{Z}))\,.
\end{equation*}
The claim follows.
\end{proof}
Notice that we are not requiring any integrability properties for the limiting integrand $\psi$ in the above Vitali convergence theorem. The integrability of $\psi$ follows directly from the proof, as we have seen.

From the above convergence theorem, it is immediate to obtain the following dominated convergence theorem:

\begin{theorem}\label{thm: dom}
Let $(M,\mathcal{M},\mu)$ be an arbitrary $\sigma$-finite measure space, and let $\mathfrak{Z}$ be an $\mathbb{F}$-Fr\'echet space. Consider a sequence of functions $\{\psi_k: M \to \mathfrak{Z}'\}_{k\in \N}$  and  $\psi: M \to \mathfrak{Z}'$
such that
\begin{enumerate}
    \item[(i)] for any $z\in \mathfrak{Z}$ and $k\in \N$ the function $x\mapsto \langle\psi_k(x),z \rangle$ is $\mu$-measurable;
    \item[(ii)] for any $z\in \mathfrak{Z}$ the function $x\mapsto \langle\psi(x),z \rangle$ is $\mu$-measurable; 
    \item[(iii)] for $\mu$-a.e. $x\in M$, $\psi_k(x)\to \psi(x)$ in $(\mathfrak{Z}',\sigma(\mathfrak{Z}',\mathfrak{Z}))$ as $k\to \infty$;
    \item[(iv)] for any $z\in \mathfrak{Z}$ there exists $0\leq g^z\in L^1(\mu)$ and $N_z\in \mathcal{M}$ with 
    $\mu(N_z)=0$ and
    \begin{equation*}
        \abs{\langle\psi_k(x),z\rangle} \leq g^z(x)\,, \quad \text{for } x\notin N_z \text{ and } k\in\N\,.
    \end{equation*}
\end{enumerate}
Then $\psi$ is $\mathfrak{Z}'$-Gelfand $\mu$-integrable, and as $k\to \infty$
\begin{equation*}
    \int_M \psi_k(x) \mu(dx) \to \int_M \psi(x) \mu(dx) \quad \text{in } (\mathfrak{Z}',\sigma(\mathfrak{Z}',\mathfrak{Z}))\,.
\end{equation*}
\end{theorem}
\begin{proof}
Clearly, for any $z\in \mathfrak{Z}$ the function $g^z$ satisfies the properties in \eqref{eq: Vitali 1} and \eqref{eq: Vitali 2}, and as a consequence the whole sequence $\langle\psi_k(\cdot),z\rangle$ must do the same. Theorem \ref{thm: Vitali} brings the thesis.
\end{proof}

Finally, we observe the our integral is well-behaved under suitable linear and continuous transformations. More precisely, we have 

\begin{proposition}
\label{prop:commute-lin-op}
Let $\psi: M \to \mathfrak{Z}'$ be $\mathfrak{Z}'$-Gelfand $\mu$-integrable. Let 
$$
T: (\mathfrak{Z}',\sigma(\mathfrak{Z}',\mathfrak{Z}))\to (\mathfrak{G}',\sigma(\mathfrak{G}',\mathfrak{G}))
$$ 
be linear and continuous, where $\mathfrak{G}$ is another $\mathbb{F}$-Fr\'echet space. Then $T\psi$ is $\mathfrak{G}'$-Gelfand $\mu$-integrable and, for any $E\in \mathcal{M}$, 
\[
T \int_E \psi(x)\,\mu(dx) =  \int_E  T\psi(x)\,\mu(dx)
\]
in $\mathfrak{G}'$.
\end{proposition}
\begin{proof}
In virtue of Proposition \ref{prop: the integral} we can define
\begin{equation*}
    \int_E  \psi(x)\,\mu(dx) \in \mathfrak{Z}'
\end{equation*}
such that for all $z\in\mathfrak{Z}$
\begin{equation*}
    \left\langle\int_E  \psi(x)\,\mu(dx) ,z\right\rangle = \int_E  \langle \psi(x),z\rangle \,\mu(dx) \,.
\end{equation*}

Since $T$ is $\sigma(\mathfrak{Z}',\mathfrak{Z}) - \sigma(\mathfrak{G}',\mathfrak{G})$ continuous, from Section \ref{sec: rapid excursion} we know, upon identifying $\mathfrak{G}$ with a subspace of $(\mathfrak{G'})^\ast$, that $T^\ast(\mathfrak{G})\subset\mathfrak{Z}$. Moreover,
\begin{equation*}
    \langle Tz',g \rangle = \langle z',T^\ast g \rangle, \quad z'\in\mathfrak{Z}', g\in\mathfrak{G}. 
\end{equation*}

But then, for any $g\in\mathfrak{G}$, by duality the function $\langle T\psi(\cdot),g \rangle = \langle \psi(\cdot),T^*g \rangle$ is an element of $L^1(\mu)$, because $T^*g\in \mathfrak{Z}$. In virtue of Proposition \ref{prop: the integral} again we can therefore define

\begin{equation*}
    \int_E  T\psi(x)\,\mu(dx) \in \mathfrak{G}'
\end{equation*}
such that for all $g\in\mathfrak{G}$
\begin{equation*}
    \left\langle\int_E  T\psi(x)\,\mu(dx) ,g\right\rangle = \int_E  \langle T\psi(x),g\rangle \,\mu(dx)\,. 
\end{equation*}
Thus, we obtain for any $g\in\mathfrak{G}$
\begin{align*}
\left\langle T\int_E\psi(x)\mu(dx),g\right\rangle&=\left\langle\int_E\psi(x)\mu(dx), T^*g\right\rangle \\
&=\int_E\langle\psi(x), T^*g\rangle\mu(dx) \\
&=\int_E\langle T\psi(x),g\rangle\mu(dx) \\
&=\left\langle\int_E T\psi(x)\mu(dx),g\right\rangle
\end{align*}
and hence the claimed result.
\end{proof}

We remark that in some situations later we will have to deal with linear operators $T: \mathfrak{Z}'\to \mathfrak{Z}'$ which are $\beta(\mathfrak{Z}',\mathfrak{Z}) - \beta(\mathfrak{Z}',\mathfrak{Z})$ continuous, a fact that seems to preclude us the possibility to apply directly the previous Proposition. The following ad hoc result will come to our aid.

\begin{proposition}\label{prop:continuity with Mackey topology}
Assume that $\beta(\mathfrak{Z}',\mathfrak{Z})=\tau(\mathfrak{Z}',\mathfrak{Z})$, namely that the strong topology coincides with the Mackey topology. Let $T: (\mathfrak{Z}',\beta(\mathfrak{Z}',\mathfrak{Z}))\to (\mathfrak{Z}',\beta(\mathfrak{Z}',\mathfrak{Z}))$ be linear and continuous. Then $T: (\mathfrak{Z}',\sigma(\mathfrak{Z}',\mathfrak{Z}))\to (\mathfrak{Z}',\sigma(\mathfrak{Z}',\mathfrak{Z}))$ is continuous.
\end{proposition}
\begin{proof}
We know from Section \ref{sec: rapid excursion} that $T:(\mathfrak{Z}',\sigma(\mathfrak{Z}',\mathfrak{Z}))\to(\mathfrak{Z}',\sigma(\mathfrak{Z}',\mathfrak{Z}))$ will be continuous if and only if for any $z\in\mathfrak{Z}$ fixed, the map $\phi_z$ so defined 
\begin{equation*}
\mathfrak{Z}'\ni z'\mapsto \langle Tz',z \rangle \in \mathbb{F}
\end{equation*}
will be $\sigma(\mathfrak{Z}',\mathfrak{Z})$ continuous. 

First of all, we notice that the evaluation map $e_z(z')=\langle z',z\rangle$ is $\beta(\mathfrak{Z}',\mathfrak{Z})$ continuous. Indeed, take an arbitrary net $z'_i\to z'$ with respect to $\beta(\mathfrak{Z}',\mathfrak{Z})$. This is equivalent to saying $p_B(z'_i-z')\to 0$ for any $B\subset\mathfrak{Z}$ bounded. The singleton $\{z\}$ is trivially bounded, and therefore $e_z(z'_i)-e_z(z')\to 0$. Since $\phi_z = e_z\circ T$, we conclude that $\phi_z$ is $\beta(\mathfrak{Z}',\mathfrak{Z})$ continuous. Trivially, it is also linear, and thus we have
\begin{equation*}
    \phi_z \in (\mathfrak{Z}',\beta(\mathfrak{Z}',\mathfrak{Z}))',\quad z\in\mathfrak{Z}.
\end{equation*}
But since the Mackey topology is consistent with the duality, and it is assumed in the present setting to coincide with $\beta(\mathfrak{Z}',\mathfrak{Z})$ , we have
\begin{equation*}
    (\mathfrak{Z}',\beta(\mathfrak{Z}',\mathfrak{Z}))'= (\mathfrak{Z}',\tau(\mathfrak{Z}',\mathfrak{Z}))' = \mathfrak{Z}.
\end{equation*}
Therefore, $\phi_z\in\mathfrak{Z}
$, and thus there must exist a unique $\hat{z}\in \mathfrak{Z}$ such that 
\begin{equation*}
    \phi_z(\cdot) = \langle \cdot,\hat{z}\rangle.
\end{equation*}
But by definition of the weak star topology, the map on the right hand side of this expression is $\sigma(\mathfrak{Z}',\mathfrak{Z}))$ continuous, and thus $\phi_z$ is so. We conclude that $T:(\mathfrak{Z}',\sigma(\mathfrak{Z}',\mathfrak{Z}))\to(\mathfrak{Z}',\sigma(\mathfrak{Z}',\mathfrak{Z}))$ is continuous.
\end{proof}

 \section{Pathwise integration}
 
The aim of this section is to make use of the Gelfand integral in Fr\'echet spaces introduced above to make sense out of integrals like
$$
\int_E\gamma(x)d\xi(x)
$$
where $\xi:\mathcal M\rightarrow\mathfrak{Z}'$ and
$\gamma:M\rightarrow\mathfrak{Z}'$. The overall goal, as well as motivation, is to include stochastic integration, where $\xi$ is random, for example defined via Brownian motion. Our integral will not follow the idea of Stieltjes' integration, but take the approach of absolute continuity with respect to a measure.   

First, let us define what we mean by an {\tt integrator}:
\begin{definition}
Given a measurable space $(M,\mathcal M)$, we say that $\xi:\mathcal M\rightarrow\mathfrak{Z}'$ is an {\it integrator} if, there exist
\begin{itemize}
    \item[(i)] a $\sigma$-finite measure $\mu$ on $(M,\mathcal M)$,
    \item[(ii)] a map $\Xi:\mathcal F(M,\mathfrak{Z}' )\rightarrow\mathcal F(M,\mathfrak{Z}')$, where $\mathcal F(M,\mathfrak{Z}')$ is the set of maps from $M$ to $\mathfrak{Z}'$,
    \item[(iii)] there exists an element $u\in\mathfrak{Z}'$ such that the map $\Xi(u):M\to\mathfrak{Z}'$ is $\mathfrak{Z}'$-Gelfand $\mu$-integrable, and for any $E\in\mathcal M$,
    $$
    \xi(E)=\int_E\Xi(u)(x)\mu(dx).
    $$
    Here, $u$ is considered as the constant map $M\ni x\mapsto u\in\mathfrak{Z}'$ and we refer to $u$ as the unit element of $\Xi$.
\end{itemize}
\end{definition}
If we are given $\xi$, we must produce a measure $\mu$ and map $\Xi$ satisfying the above (i)-(iii). Implicitly, in finding $\Xi$, we must also find a unit element for it. On the other hand, if we specify a measure $\mu$ and a map $\Xi$ together with an element $v\in\mathfrak{Z}'$ for which $\Xi(v):M\to\mathfrak{Z}'$ is $\mathfrak{Z}'$-Gelfand $\mu$-integrable, we can {\it define} an integrator $\xi$ for the measure $\mu$ and map $\Xi$ with unit $v$ by
$$
\xi(E):=\int_E\Xi(v)(x)\mu(dx).
$$
In stochastic integration, we are typically given $\xi$.

For an integrator, we introduce the class of {\it integrands} and define the integral with respect to $\xi$ as follows:
\begin{definition}
Let $\xi$ be an integrator with measure $\mu$ and map $\Xi$. We say that $\gamma:M\rightarrow\mathfrak{Z}'$ is {\tt integrable} with respect to $\xi$ if $\Xi(\gamma)$ is $\mathfrak{Z}'$-Gelfand $\mu$-integrable, and denote the set of such $\gamma$ as $\mathcal I(\xi)$. If $\gamma\in\mathcal I(\xi)$, then for any $E\in\mathcal M$, 
$$
\int_E\gamma(x)d\xi(x):=\int_E\Xi(\gamma)(x)\mu(dx)\in\mathfrak{Z}',
$$
where the right-hand side is interpreted as the Gelfand integral of Section \ref{sec:Gelfand} using $\psi(x):=\Xi(\gamma)(x)$.  
\end{definition}
Choosing $\gamma(x):=u$, with $u$ being the unit element of $\Xi$, we readily get from the definitions above that $\gamma$ is integrable with respect to $\xi$ and 
$\xi(E)=\int_E\Xi(u)(x)\mu(dx)$, as expected. 

The set $\mathcal F(M,\mathfrak{Z}')$ is a vector space over the field $\mathbb F$. The next lemma shows linearity of the defined integral when $\Xi$ is linear. 
\begin{lemma}
Let $\xi$ be an integrator with measure $\mu$ and map $\Xi$. Assume that $\Xi$ is linear, i.e., that for each $a,b\in\mathbb F$ and $\gamma_1,\gamma_2:M\rightarrow\mathfrak{Z}'$, $\Xi(a\gamma_1+b\gamma_2)=a\Xi(\gamma_1)+b\Xi(\gamma_2)$. Then $a\gamma_1+b\gamma_2\in\mathcal I(\xi)$ for all $a,b\in\mathbb F$ whenever $\gamma_1,\gamma_2\in\mathcal I(\xi)$. Furthermore, for all $E\in\mathcal M$
$$
\int_Ea\gamma_1(x)+b\gamma_2(x)d\xi(x)=a\int_E\gamma_1(x)d\xi(x)+b\int_E\gamma_2(x)d\xi(x).
$$
\end{lemma}
\begin{proof}
By the linearity of $\Xi$, we find for every $z\in\mathfrak{Z}$ that 
$$
\langle\Xi(a\gamma_1+b\gamma_2)(\cdot),z\rangle=a\langle\Xi(\gamma_1)(\cdot),z\rangle+b\langle\Xi(\gamma_2)(\cdot),z\rangle\in L^1(\mu).
$$
The linearity of the integral follows again from the linearity of $\Xi$ and the linearity of the Gelfand integral.  
\end{proof}
Linearity of an integral with respect to the integrand is a desirable property, and we shall from now on mostly focus on linear maps $\Xi$ of the form discussed below. 

We also recall the continuity of the integral
with respect to the domain of integration provided by Proposition \ref{prop: weak star continuity wrt FN metric}.

\subsection{Pathwise integration in algebras}

Let us now consider a special case relevant for stochastic integration. 
Suppose $(\mathfrak{Z}',\sigma(\mathfrak{Z}',\mathfrak{Z}))$ is a locally convex algebra, where additionally the multiplication, denoted by $\diamond$, is assumed to be associative and with unit $1\in\mathfrak{Z}'$. Note that $\diamond$ is separately $\sigma(\mathfrak{Z}',\mathfrak{Z})-\sigma(\mathfrak{Z}',\mathfrak{Z})$ continuous by assumption, and by this we mean that for any net $(\gamma_i)_i$ in $\mathfrak{Z}'$ converging to $\gamma\in\mathfrak{Z}'$ with respect to $\sigma(\mathfrak{Z}',\mathfrak{Z})$, it holds that the net $(\langle \theta\diamond\gamma_i,z\rangle)_i$ converges to 
$\langle\theta\diamond\gamma,z\rangle$ for any $\theta\in\mathfrak{Z}'$ and any $z\in\mathfrak{Z}$; and a completely analogous statement holds for left-continuity. 

We can define integrators by the following. We first assume that for a given $\xi:\mathcal M\to\mathfrak{Z}'$ there exists a $\sigma$-finite measure $\mu$ on $(M,\mathcal M)$ and a map
$\dot{\xi}:M\rightarrow\mathfrak{Z}'$ such that $\dot{\xi}$ is $\mathfrak{Z}'$-Gelfand $\mu$-integrable, and
$$
\xi(E)=\int_E\dot{\xi}(x)\mu(dx) \in \mathfrak{Z}'
$$
for any $E\in\mathcal M$. Next, define the map 
\begin{equation}
\label{def:map-for-unital-algebra}
    \Xi(\gamma)(x)=\gamma(x)\diamond\dot{\xi}(x)
\end{equation}
This defines a linear map on $\mathcal F(M,\mathfrak{Z}')$ into itself. Moreover, $\Xi(1)=\dot{\xi}$, and therefore $1$ is the unit element of $\Xi$ since
$$
\xi(E)=\int_E\Xi(1)(x)\mu(dx).
$$
Here, $1$ is considered as the constant map $M\ni x\mapsto 1\in\mathfrak{Z}'$. This shows that $\xi$ is an integrator with measure $\mu$ and linear map $\Xi$ given in \eqref{def:map-for-unital-algebra}.  
We may not have a commutative product $\diamond$, and thus we could alternatively define $\Xi(\gamma)(x)=\dot{\xi}(x)\diamond\gamma(x)$.
If $\gamma\in\mathcal I(\xi)$, that is, if $\gamma$ is such that
$\gamma\diamond\dot{\xi}$ is $\mathfrak{Z}'$-Gelfand $\mu$-integrable, then 
$$
\int_E\gamma(x)d\xi(x):=\int_E\gamma(x)\diamond\dot{\xi}(x)\mu(dx)
$$
for $E\in\mathcal M$. 

\begin{example}
\label{ex-hida-int}
As we recall from the the introduction to Hida spaces in Section \ref{sec: rapid excursion},  
we can select $\mathfrak{Z}$ to be the space of Hida smooth random variables $(\mathcal S)$, and $\mathfrak{Z}'$ to be the Hida distribution space $(\mathcal S)'$. This space is endowed with the inductive limit topology $\tau_{\texttt{ind}}$, which coincides both with the strong topology $\beta((\mathcal S)',(\mathcal S)) $ and the Mackey topology $\tau((\mathcal S)',(\mathcal S))$ \cite[Page 482]{HKPS}. We observe that $(\mathcal S)'$ is also a Montel space, being the strong dual of a Montel space \cite[Page 147]{Schaefer}. 

On $(\mathcal S)'$ we have the Wick product $\diamond$ which, in view of 
\cite[Corollary 4.22]{HKPS}  
$$
\Vert F\diamond G\Vert_{2,-r}\leq C\Vert F\Vert_{2,-p}\Vert G\Vert_{2,-q}
$$
where $r=-p\vee q-\alpha$, $\alpha>1/2$, is jointly continuous with respect to $\tau_{\texttt{ind}}$. It is commutative, distributive and associative (see \cite[Lemma 2.4.5]{HOUZ}). A double application of Proposition \ref{prop:continuity with Mackey topology} enables us then to conclude that actually $\diamond$ is separately $\sigma((\mathcal S)',(\mathcal S))-\sigma((\mathcal S)',(\mathcal S))$ continuous too. We can define an integral with respect to Brownian motion $\xi:=B$ as
$$
\int_E \gamma(s)dB(s):=\int_E \gamma(s)\diamond W(s)ds
$$
where $\gamma$ is a mapping from $E$ to $(\mathcal S)'$ such that $s\mapsto\gamma(s)\diamond W(s)$ is integrable on $E$. Naturally, $E$ is some measurable subset of $\mathbb R$ and we recall $W(s)$ to be white noise.
Notice that the unit element is the trivial random variable $1\in (L^2)$ (see Example 2.5.10 in \cite{HOUZ}). Usually, the left-hand side above is referred to as the Skorohod integral, and we denote it by $\int_E\gamma(s)\delta B(s)$ to distinguish it from the special case of Ito integral. Notice that by Corollary \ref{corr: strong continuity wrt FN metric} we have strong continuity of $t\mapsto\int_0^t\gamma(s)\delta B(s)\in (\mathcal S)'$.
\end{example}

As the next result shows, we can pull constants out of the integral:
\begin{lemma}
Assume $\theta\in\mathfrak{Z}'$.
Then, for any $\gamma\in\mathcal I(\xi)$ we have that $\theta\diamond \gamma\in\mathcal I(\xi)$ and
$$
\theta\diamond\int_E\gamma(x)d\xi(x)=\int_E(\theta\diamond\gamma(x))d\xi(x).
$$
\end{lemma}
\begin{proof}
Introduce the linear map
$$
T:\mathfrak{Z}'\rightarrow\mathfrak{Z}',\qquad T\gamma=\theta\diamond\gamma,
$$
which is $\sigma(\mathfrak{Z}',\mathfrak{Z})-\sigma(\mathfrak{Z}',\mathfrak{Z})$ continuous by the assumption on the product on $\mathfrak{Z}'$. Then, from Prop. \ref{prop:commute-lin-op} as well as by the associativity of the product, we find for $\gamma\in\mathcal I(\xi)$ that $T\gamma\in\mathcal I(\xi)$ and
\begin{align*}
\theta\diamond\int_E\gamma(x)d\xi(x)&=\int_E T(\Xi(\gamma)(x))\mu(dx)\\&=\int_E\theta\diamond\gamma(x)\diamond\dot{\xi}(x)\mu(dx)\\&=\int_E(\theta\diamond\gamma(x))d\xi(x).
\end{align*}
\end{proof}
The above result is known for the Skorohod integral in White Noise Analysis, see Holden {\it et al.} \cite[Cor. 2.5.12]{HOUZ}. 

From the definition of $\Xi(\gamma)(x)=\gamma(x)\diamond\dot{\xi}(x)$, we observe that we can express $\Xi$ as a linear operator $\Xi_x:\mathfrak{Z}'\rightarrow\mathfrak{Z}'$ for each fixed $x\in M$ as
\begin{equation}
\label{lin-op-Xi}
\Xi(\gamma)(x):=\Xi_x(\gamma(x))
\end{equation}
where
$$
\mathfrak{Z}'\ni\eta\mapsto\Xi_x(\eta)=\eta\diamond\dot{\xi}(x) 
$$
Hence, $M\ni x\rightarrow\Xi_x\in L(\mathfrak{Z}',\sigma)$, where $L(\mathfrak{Z}',\sigma)$ is the space of linear and $\sigma(\mathfrak{Z}',\mathfrak{Z})-\sigma(\mathfrak{Z}',\mathfrak{Z})$ continuous operators on $\mathfrak{Z}'$.
Let us now study operators $\Xi$ which has a representation given by a $\Xi_x\in L(\mathfrak Z',\sigma)$ as in \eqref{lin-op-Xi} in the general situation.

For maps $\Xi_x$ being linear, we can readily show a Fubini result on the pathwise Gelfand integrals:
\begin{proposition}
\label{prop:fubini-gen-stoch}
Assume, for $i=1,2$, that $\xi_i$ are integrators with measures $\mu_i$ on the
spaces $(M_i,\mathcal M_i)$ with maps $\Xi_i$. For $\Xi_1$ and $\Xi_2$, assume there exist $\Xi_{1,x},\Xi_{2,y}\in L(\mathfrak{Z}',\sigma)$, $x\in M_1,y\in M_2$, resp., such that \eqref{lin-op-Xi} hold. Consider a mapping 
$$
M_1\times M_2\ni(x,y)\mapsto \gamma(x,y)\in\mathfrak Z',
$$
where $\gamma(x,\cdot)\in\mathcal I(\xi_2)$ for any $x\in M_1$, $\gamma(\cdot,y)\in\mathcal I(\xi_1)$ for any $y\in M_2$, and $\int_{E_2}\gamma(\cdot,y)d\xi_2(y)\in\mathcal I(\xi_1)$ and $\int_{E_1}\gamma(x,\cdot)d\xi_1(x)\in\mathcal I(\xi_2)$ for any $E_i\in\mathcal M_i$, $i=1,2$. 

If $\Xi_{1,x}$ and $\Xi_{2,y}$ commutes, that is, 
$$
\Xi_{1,x}(\Xi_{2,y}(\gamma(x,y)))=\Xi_{2,y}(\Xi_{1,x}(\gamma(x,y))), \quad (x,y)\in M_1\times M_2
$$
with $(x,y)\mapsto \Xi_{1,x}(\Xi_{2,y}(\gamma(x,y)))$ being 
$\mathfrak{Z}'$-Gelfand $\mu_1\times\mu_2$-integrable, then
$$
\int_{E_1}\int_{E_2}\gamma(x,y)d\xi_2(x)d\xi_1(y)=\int_{E_2}\int_{E_1}\gamma(x,y)d\xi_1(y)d\xi_2(x)
$$
for any $E_i\in\mathcal M_i$, $i=1,2$. 
\end{proposition}
\begin{proof}
By the assumptions, we find 
\begin{align*}
\int_{E_1}\int_{E_2}\gamma(x,y)d\xi_2(y)d\xi_1(x)&=\int_{E_1}\int_{E_2}\Xi_{2,y}(\gamma(x,y))\mu_2(dy)d\xi_1(x) \\
&=\int_{E_1}\Xi_{1,x}\left(\int_{E_2}\Xi_{2,y}(\gamma(x,y))\mu_2(dy)\right)\mu_1(dx) \\
&=\int_{E_1}\int_{E_2}\Xi_{1,x}(\Xi_{2,y}(\gamma(x,y))\mu_2(dy)\mu_1(dx)
\end{align*}
where we appealed to Prop. \ref{prop:commute-lin-op} in the last equality. 

Analogously, 
$$
\int_{E_2}\int_{E_1}\gamma(x,y)d\xi_1(x)d\xi_2(y)=\int_{E_2}\int_{E_1}\Xi_{2,y}(\Xi_{1,x}(\gamma(x,y))\mu_1(dx)\mu_2(dy). 
$$
From the commutativity and joint integrability assumption of $\Xi_{1,x}$ and $\Xi_{2,y}$, the result follows from the Fubini Theorem \ref{thm:fubini} of Gelfand integrals. 
\end{proof}
Apart from natural integrability conditions, the Fubini result above rests on the assumptions that $\Xi_{1,x}$ and $\Xi_{2,y}$ are linear and continuous operators which commute. We can use Prop. \ref{prop:fubini-gen-stoch} to define iterative stochastic integrals in the Hida space. Recall the set-up in Example \ref{ex-hida-int}. As $(\gamma(s,t)\diamond W(s))\diamond W(t)=(\gamma(s,t)\diamond W(t))\diamond W(s)$, we invoke under natural integrability conditions on $(s,t)\mapsto\gamma(s,t)$ that
$$
\int_{E_1}\int_{E_2}\gamma(s,t)\delta B(s)\delta B(t)=\int_{E_2}\int_{E_1}\gamma(s,t)\delta B(t)\delta B(s)
$$
for $E_1,E_2$ being measurable subsets of $\mathbb R$. To have a Fubini-result for Ito integration, we must further 
invoke measurability conditions like adaptedness.

\subsection{The Gelfand derivative}

We have the following duality representation for the integral:
\begin{lemma}
\label{lemma:gradient}
Suppose $\xi$ is an integrator with map $\Xi$, and assume that there exists for every $x\in M$
a $\Xi_x\in L(\mathfrak{Z}',\sigma)$ such that for any $\gamma\in\mathcal I(\xi)$, $\Xi(\gamma)(x)=\Xi_x(\gamma(x))$. Then,
$$
\left\langle\int_E\gamma(x)d\xi(x),z\right\rangle=\int_E\langle\gamma(x),\Xi_x^*(z)\rangle\mu(dx)
$$
for any $E\in\mathcal{M}$ and any $z\in\mathfrak{Z}$, where $\Xi_x^*:(\mathfrak{Z}')^*\rightarrow(\mathfrak{Z}')^*$ is the dual of $\Xi_x$.  
\end{lemma}
\begin{proof}
Observe that $\gamma\in\mathcal I(\xi)$ means with the definition of $\Xi$ that $x\mapsto\langle\Xi_x(\gamma(x)),z\rangle$ is integrable with respect to $\mu$ on $M$. 
Next, since $\sigma(\mathfrak{Z}',\mathfrak{Z})$ is consistent with the duality $\langle\mathfrak{Z}',\mathfrak{Z}\rangle$, it holds $(\mathfrak{Z}',\sigma(\mathfrak{Z}',\mathfrak{Z}))'=\mathfrak{Z}$. Moreover, the dual operator $\Xi_x^*$ satisfies $\Xi_x^*(\mathfrak{Z})\subset\mathfrak{Z}$, because of the continuity of $\Xi_x$, and the result follows. 
\end{proof}
Hence, $\Xi_x^*$ may be interpreted as a "gradient" for the integral. If $\Xi_x(\eta)=\eta\diamond\dot{\xi}(x)$, we may call the dual operator $\Xi_x^*$ the {\it Gelfand derivative} as we obtain an operator representation of the integrals in Lemma \ref{lemma:gradient} generalizing the Hida-Malliavin derivative in the Hida space setting. Indeed, if $\mathfrak{Z}'$ is the 
Hida distribution space $(\mathcal S)'$, then for a Skorohod integrable process $\gamma$ on some interval $[0,t]$ it holds that 
$$
\mathbb E\left[\int_0^t\gamma(s)\delta B(s)z\right]=\mathbb E\left[\int_0^t\gamma(s) D_sz ds\right]
$$
for any Malliavin differentiable $z\in\mathcal D_{1,2}\subset(L^2)$. This is the duality relation between the Skorohod integral and the Malliavin derivative, which is used as the definition of Skorohod integration in Nualart \cite[Def. 1.3.1]{Nualart-book} as an adjoint operator.

Consider now the case where $\mathfrak{Z}$ is an algebra for a product $\cdot$, i.e., $y\cdot z\in\mathfrak{Z}$ for $y,z\in\mathfrak{Z}$. We assume that $\cdot$ is associative and a continuous operator in $\mathfrak{Z}$. 
\begin{remark}
In the Hida space $(\mathcal S)$, this is the pointwise product for random variables, and thus not the same as the Wick product.
\end{remark}

We can define an operator $T\in L(\mathfrak{Z}',\sigma)$ as follows.
\begin{definition}
We fix a $\theta\in\mathfrak{Z}$, and define $T\gamma=\theta\cdot\gamma$ for every $\gamma\in\mathfrak{Z}'$ by the relation
$$
\langle\theta\cdot\gamma,z\rangle:=\langle\gamma,\theta\cdot z\rangle
$$
for any $z\in\mathfrak{Z}$. 
\end{definition}

Let us introduce {\it derivative operators} on $\mathfrak{Z}$:
\begin{definition}
We say that $\Gamma:\mathfrak{Z}\rightarrow\mathfrak{Z}$ is a {\it derivative operator} if 
$$
\Gamma(y\cdot z)=y\cdot\Gamma(z)+\Gamma(y)\cdot z
$$
for any $y,z\in\mathfrak{Z}$.
\end{definition}

In Benth \cite{Benth}, a derivative operator in the Hida distribution space $(\mathcal S)'$ is defined as follows. For an $f\in\mathcal S(\mathbb R)$, recall the the random variable $W(f)(\omega):=(f,\omega)$, with $(\cdot,\cdot)$ being the dual pairing between $\mathcal S(\mathbb R)$ and $\mathcal S'(\mathbb R)$.  Since 
$W(f)\in(\mathcal S)$, we can define the operator $\mathcal D_f:(\mathcal S)'\rightarrow(\mathcal S)'$ as
\begin{equation}
\label{def:gen-mall-der}
\mathcal D_f\Phi:=\Phi\cdot W(f)-\Phi\diamond W(f).
\end{equation}
As the Hida test function space $(\mathcal S)$ is closed under both the pointwise and Wick product, it holds that 
$\mathcal D_f:(\mathcal S)\rightarrow(\mathcal S)$. Moreover, as we show in the appendix, $\mathcal D_f$ is a derivative operator. If $\Phi$ is a Malliavin differentiable random variable, then it is shown in Benth \cite[Thm. 3.3]{Benth} that $\mathcal D_f$ coincides with the Malliavin derivative in the sense that 
$$
\mathcal D_f\Phi=\int_{\mathbb R} (D_t\Phi) f(t)dt.
$$
This links $\mathcal D$ to the Malliavn derivative, yielding a generalization to an operator on $(\mathcal S)'$ which is a derivative on $(\mathcal S)$. We note in passing that $\mathcal D_f$ in fact is a derivative operator on $(\mathcal S)'$ under the Wick product (see Benth \cite[Thm. 4.1]{Benth}). By denoting $f^t(x):=f(x-t)$, the shift of $f\in\mathcal S(\mathbb R)$, consider the {\it smoothed white noise} $\widetilde{W}(t):=( f^t,\cdot)\in(\mathcal S)$. For the operator $\Xi_t(\phi):=\Phi\diamond\widetilde{W}(t)$, we define the stochastic integral
\begin{equation}
\label{def:smoothed-wick-integral}
\int_0^t\Phi(s)\delta\widetilde{B}(s):=\int_0^t\Xi_s(\Phi(s))ds=\int_0^t\Phi(s)\diamond \widetilde{W}(s)ds
\end{equation}
for integrable Hida distribution processes $\Phi(s)$. It holds for any $\phi\in(\mathcal S)$ that (see Appendix for a proof),
\begin{equation}
    \langle\int_0^t\Phi(s)\delta\widetilde{B}(s),\phi\rangle=\int_0^t\langle\Phi(s),\mathcal D_{f^s}(\phi)\rangle ds.
\end{equation}
We therefore have an example of a derivative operator 
$\Xi_t^*:=\mathcal D_{f^t}$. As a last remark, define the integral 
\begin{equation}
\label{def:smoothed-pointwise-integral}
\int_0^t\Phi(s)\delta^{\circ}\widetilde{B}(s):=\int_0^t\Phi(s)\cdot\widetilde{W}(s)ds
\end{equation}
Then, by the definition of $\mathcal D_{f^t}$ it holds that 
\begin{equation}
\label{eq:relation-strat-skorohod}
\int_0^t\Phi(s)\delta^{\circ}\widetilde{B}(s)=\int_0^t\mathcal D_{f^s}\Phi(s)ds+\int_0^t\Phi(s)\delta\widetilde{B}(s)
\end{equation}
for all Hida distribution processes $\Phi$ for which the above integrals are well-defined. Recalling that $\mathcal D_{f_t}$ can be interpreted as a Malliavin detivative for Hida distributions, we have a connection between stochastic integrals defined by pointwise and Wick multiplication with a smoothed noise resembling the relationship between the Stratonovich and Skorohod integral, see Thm. 3.1.1 and the following remark on page 152 in Nualart \cite{Nualart-book}. This lends itself to the interpretation of the integral in \eqref{def:smoothed-pointwise-integral} as a Stratonovich-type of stochastic integral.

Let us return to the general considerations. Under the assumption that $\Xi_x^*$ introduced in Lemma \ref{lemma:gradient} is a derivative operator, we find the following {\it integration by parts} formula for integrals:

\begin{proposition}
Let $\theta\in\mathfrak{Z}$. Suppose that $\xi$ is an integrator with map $\Xi$ such that there exists for every $x\in M$ a $\Xi_x\in L(\mathfrak{Z}',\sigma)$ such that for any $\gamma\in\mathcal I(\xi)$, $\Xi(\gamma)(x)=\Xi_x(\gamma(x))$. If the dual $\Xi_x^*$  
restricted to $\mathfrak{Z}$
is a derivative operator, then 
\begin{equation}
\int_E\theta\cdot\gamma(x)d\xi(x)=\theta\cdot\int_E\gamma(x)d\xi(x)-\int_E\Xi_x^*(\theta)\cdot\gamma(x)\mu(dx)
\end{equation}
\end{proposition}
\begin{proof}

First we observe that by 
assumptions, $\theta\cdot\Xi^*_x(z)\in\mathfrak Z$ and
$\langle\gamma(x),\theta\cdot\Xi_x^*(z)\rangle\in L^1(\mu)$ for all $z\in\mathfrak Z$. By definition of the multiplication operator and duality, we compute
\begin{align*}
\langle\gamma(x),\theta\cdot\Xi_x^*(z)\rangle&=\langle\theta\cdot\gamma(x),\Xi_x^*(z)\rangle=\langle\Xi_x(\theta\cdot\gamma(x)),z\rangle.
\end{align*}
As this holds for any $z\in\mathfrak Z$,
$\theta\cdot\gamma\in\mathcal I(\xi)$. 

We refer to Proposition \ref{prop:commute-lin-op}, Lemma \ref{lemma:gradient} and the derivative property of $\Xi_x^*$ to do the following calculation: for any $z\in\mathfrak{Z}$,
\begin{align*}
    \langle\int_M\theta\cdot\gamma(x)d\xi(x),z\rangle&=\int_M\langle\Xi_x(\theta\cdot\gamma(x)),z\rangle\mu(dx) \\
    &=\int_M\langle\gamma(x),\theta\cdot\Xi_x^*(z)\rangle\mu(dx) \\
    &=\int_M\langle\gamma(x),\Xi_x^*(\theta\cdot z)-\Xi_x^*(\theta)\cdot z\rangle\mu(dx) \\
    &=\int_M\langle\Xi_x(\gamma(x)),\theta\cdot z\rangle\mu(dx)-\int_M\langle\Xi_x^*(\theta)\cdot\gamma(x),z\rangle\mu(dx) \\
    &=\langle\int_M\Xi(\gamma)(x)\mu(dx),\theta\cdot z\rangle-\langle\int_M\Xi_x^*(\theta)\cdot\gamma(x)\mu(dx),z\rangle \\
    &=\langle\theta\cdot\int_M\gamma(x)d\xi(x)-\int_M\Xi_x^*(\theta)\cdot\gamma(x)\mu(dx),z\rangle.
\end{align*}
The claim follows. 
\end{proof}
Thus, we have established an {\it integration-by-parts} formula similar to the one in Malliavin Calculus (where $\Xi_x^*$ is interpreted as the Malliavin derivative), see Nualart \cite[Eq. (1.49) on page 40]{Nualart-book}.

\subsection{Stability for pathwise integrals}

Let us show some stability results for the integral in the general case.
Using dominated convergence we can show the following stability result with respect to the integrators:
\begin{proposition}
\label{prop:stability-integrator}
Suppose that $(\xi_k)_{k\in\mathbb N}$ is a sequence of integrators with common measure $\mu$ and maps $(\Xi_{k})_{k\in\mathbb N}$ where $\bigcap \mathcal I(\xi_k)\neq\emptyset$. Assume there exist a map $\Xi:\mathcal F(M,\mathfrak{Z}')\rightarrow\mathcal F(M,\mathfrak{Z}')$ and a unit element $u\in\mathfrak{Z}'$ such that $\Xi(u)$ is $\mathfrak{Z}'$-Gelfand $\mu$-integrable,
and where the following holds for any $\gamma\in\bigcap\mathcal I(\xi_k)$:
\begin{enumerate}
    \item for any $z\in\mathfrak{Z}$, $x\mapsto\langle\Xi(\gamma)(x),z\rangle$ is $\mu$-measurable,
    \item for $\mu-a.e.$ $x\in M$, $\Xi_k(\gamma)(x)\rightarrow\Xi(\gamma)(x)$ in 
    $(\mathfrak{Z}',\sigma(\mathfrak{Z}',\mathfrak{Z}))$,
    \item for any $z\in\mathfrak{Z}$ there exists a $0\leq g^z\in L^1(\mu)$ such that $\vert\langle\Xi_k(\gamma)(x),z\rangle\vert\leq g^z(x)$ $\mu-a.e.\, x\in M$.
\end{enumerate}
Then $\xi:\mathcal M\rightarrow\mathfrak{Z}'$ defined as 
$$
\xi(E):=\int_E\Xi(\gamma)(x)\mu(dx)
$$
is an integrator, $\bigcap\mathcal I(\xi_k)\subset\mathcal I(\xi)$ and
$$
\int_M\gamma(x)d\xi_k(x)\rightarrow\int_M\gamma(x)d\xi(x)
$$
in $(\mathfrak{Z}',\sigma(\mathfrak{Z}',\mathfrak{Z}))$.
\end{proposition}
\begin{proof}
First, we observe that $\xi$ defines an integrator with measure $\mu$ and map $\Xi$ by the assumptions. Next, choose a $\gamma\in\bigcap\mathcal I(\xi_k)$, and set $\psi_k(x):=\Xi_k(\gamma)(x)$ and $\psi(x):=\Xi(\gamma)(x)$. Notice that $\psi_k$ is $\mathfrak{Z}'$-Gelfand $\mu$-integrable, since $\gamma\in\mathcal I(\xi_k)$ for every $k\in\mathbb N$. Invoking the assumptions, it follows from Thm. \ref{thm: dom} that $\psi$ is $\mathfrak{Z}'$-Gelfand $\mu$-integrable, and $\int_M\psi_k(x)\mu(dx)\rightarrow\int_M\psi(x)\mu(dx)$ in $(\mathfrak{Z}',\sigma(\mathfrak{Z}',\mathfrak{Z}))$. In other words, $\Xi(\gamma)(\cdot)$ is $\mathfrak{Z}'$-Gelfand $\mu$-integrable, which implies $\gamma\in\mathcal I(\xi)$, and 
$$
\int_M\gamma(x)d\xi_k(x)=\int_M\Xi_k(\gamma)(x)\mu(dx)\rightarrow \int_M\Xi(\gamma)(x)\mu(dx)=\int_M\gamma(x)d\xi(x) 
$$
in $(\mathfrak{Z}',\sigma(\mathfrak{Z}',\mathfrak{Z}))$. This yields the result. 
\end{proof}
Recall the smoothed stochastic integral defined in 
\eqref{def:smoothed-wick-integral} for the Hida distribution space $(\mathcal S)'$. It is well-known that one can choose a sequence of functions $(f_k)\subset\mathcal S(\mathbb R)$ such that $f_k\rightarrow\delta_0\in\mathcal S'(\mathbb R)$, where $\delta_0$ is the Dirac-$\delta$ function and the convergence is in distributional sense. It follows that $\widetilde{W}_k(t):=\langle f_k^t,\cdot\rangle\rightarrow W(t)$, where $W(t)$ is the white noise process in $(\mathcal S)'$. Indeed, from the above Prop. \ref{prop:stability-integrator} we find that 
$$
\int_0^t\Phi(s)\delta\widetilde{B}_k(s)\rightarrow\int_0^t\Phi(s)\delta B(s)
$$
for a suitable class of integrands $\Phi$. 


We have an analogous result to Prop. \ref{prop:stability-integrator}
concerning stability with respect to the integrands:
\begin{proposition}
Suppose that $(\gamma_k)_{k\in\mathbb N}$ is a sequence in $\mathcal I(\xi)$ for an integrator $\xi$ with measure $\mu$ and map $\Xi$. Assume that there exists a $\gamma:M\rightarrow\mathfrak{Z}'$ such that the following holds:
\begin{enumerate}
    \item for any $z\in\mathfrak{Z}$, $x\mapsto\langle\Xi(\gamma)(x),z\rangle$ is $\mu$-measurable,
    \item for $\mu-a.e$ $x\in M$, $\Xi(\gamma_k)(x)\rightarrow\Xi(\gamma)(x)$ in $(\mathfrak{Z}',\sigma(\mathfrak{Z}',\mathfrak{Z}))$,
    \item for any $z\in\mathfrak{Z}$ there exists a $0\leq g^z\in L^1(\mu)$ such that $\vert\langle\Xi(\gamma_k)(x),z\rangle\vert\leq g^z(x)$ $\mu-a.e.\, x\in M$.
\end{enumerate}
Then, $\gamma\in\mathcal I(\xi)$ and
$$
\int_M\gamma_k(x)d\xi(x)\rightarrow\int_M\gamma(x)d\xi(x)
$$
in $(\mathfrak{Z}',\sigma(\mathfrak{Z}',\mathfrak{Z}))$. 
\end{proposition}
\begin{proof}
Set $\psi_k(x):=\Xi(\gamma_k)(x)$ and $\psi(x):=\Xi(\gamma)(x)$. As $\gamma_k\in\mathcal I(\xi)$, we find that $\psi_k$ is $\mathfrak{Z}'$-Gelfand $\mu$-integrable for every $k\in\mathbb N$. Invoking the assumptions, it follows from Thm. \ref{thm: dom} that $\psi$ is $\mathfrak{Z}'$-Gelfand $\mu$-integrable and 
$$
\int_M\Xi(\gamma_k)(x)\mu(dx)\rightarrow\int_M\Xi(\gamma)(x)\mu(dx).
$$
The claims follow.
\end{proof}

Let us go back to a case study, again relevant to stochastic integration. If $\mathfrak{Z}$ has a multiplication operator, denoted $\cdot$, that enjoys some mild and natural continuity properties (e.g., being an algebra with this multiplication operator), we can define for any $x\in\mathfrak{Z}$ a linear operator mapping $\mathfrak{Z}'$ into $\mathfrak{Z}'$ by
$$
\mathfrak Z'\ni\eta\mapsto\Xi_x(\eta):=\eta\cdot\dot{\xi}(x)
$$
where in this case $\dot{\xi}:M\rightarrow\mathfrak{Z}$. $\Xi_x(\eta)$ acts linearily on $\mathfrak{Z}$ by
$$
\langle\Xi_x(\eta),z\rangle\equiv \langle\eta\cdot\dot{\xi}(x),z\rangle:=\langle\eta,z\cdot\dot{\xi}(x)\rangle
$$
Notice that $z\cdot\dot{\xi}(x)\in\mathfrak{Z}$. By this we can introduce the integral
$$
\int_E\gamma(x)d\xi(x)=\int_E\gamma(x)\cdot\dot{\xi}(x)\mu(dx)
$$
where $\gamma:M\to\mathfrak{Z}'$ is such that $x\to \Xi_x(\gamma(x))$ is $\mathfrak{Z}'$-Gelfand $\mu$-integrable.

Observe that if $\mathfrak{Z}\subset\mathfrak{Z}'$, one could also use the product $\diamond$, which is not necessarily equal to the product $\cdot$. In Hida distribution spaces, $\diamond$ is the Wick product while $\cdot$ is the pointwise (i.e., $\omega$-wise) product of random variables.


\subsection{A discussion of integration with respect to Volterra processes in Hida distribution space}

Consider a Volterra process
\begin{equation}
    Y(t)=\int_0^t g(t,s)dB(s)
\end{equation}
where $B$ is a Brownian motion and $g$ is a measurable deterministic function mapping from $\mathbb R^2$ into $\mathbb R$. For each $t>0$, we suppose that
$s\mapsto g(t,s)$ is square-integrable. Our concern here is the definition a stochastic Volterra integral
\begin{equation}
\int_0^t \Phi(s)dY(s)
\end{equation}
for a suitable class of stochastic processes $\Phi$.

Alos, Mazet, Nualart \cite{AMN} propose to define the stochastic Volterra integral as 
\begin{equation}
\label{eq:volterra-alos}
    \int_0^t\Phi(s)dY(s)=\int_0^t\mathcal K_g(\Phi)(t,s)\delta B(s)
\end{equation}
where the operator $\mathcal K_g$ is 
given by 
\begin{equation}
\label{def:K_g}
    \mathcal K_g(\Phi)(t,s)=\Phi(s)g(t,s)+\int_s^t(\Phi(u)-\Phi(s))g(du,s)
\end{equation}
This operator is well-defined for functions $g$ which is of finite variation in the first variable, and functions $f$ where paths are sufficiently regular to allow for integration with respect to $g(du,s)$. Thus, the space of integrands $\Phi$ are those processes for which $\mathcal K_g(\Phi)(t,s)$ is well-defined and Skorohod integrable (see Alos {\it et al.} \cite{AMN} for more details). 

From this definition, we can extend the stochastic Volterra integral to the Hida distribution space as follows: Interpret the operator $\mathcal K_g$ as a Gelfand integral on $(\mathbb R,\mathcal B_{\mathbb R})$, equipped with the measure induced by $g(du,s)$. Then, define the operator $\Xi_{t,s}(\Phi)$ as
$$
\Xi_{t,s}(\Phi)=\mathcal K_g(\Phi)(t,s)\diamond W(s).
$$
Not surprisingly, the operator $\Xi$ will depend on the integration domain. Additionally, from the operator $\mathcal K_g$, we need the path of $\Phi$ from $s$ to $t$, and not only a point evaluation in current time $s$. This also shows the need for using an anticipative integral definition yielded by the Skorohod integral. 

Let us apply the White Noise framework to provide a motivation for the definition of Alos {\it et al.} \cite{AMN}. Using the representation 
$$
\int_0^tg(t,s)dB(s)=\int_0^tg(t,s)W(s)ds
$$
we find by a formal calculation,

\begin{align*}
Y'(t)&=\frac{d}{dt}\int_0^tg(t,s)W(s)ds=g(t,t)W(t)+\int_0^t\partial_t g(t,u)W(u)du
\end{align*}
This derivation requires that $g$ is differentiable in the first variable, with sufficient integrability of the derivative, and that the diagonal $g(t,t)$ is finite. These conditions may not hold, for example, for the Liouville-representation of fractional Brownian motion, or certain BSS-processes with gamma kernels in ambit stochastics (see Barndorff-Nielsen, Benth and Veraart \cite{ambit-book}). We introduce the stochastic Volterra integral by resorting to the Wick product and the formal derivative of $Y$:
\begin{equation}
\label{eq:volterra-wick-def}
   \int_0^t\Phi(s)dY(s):=\int_0^t\Phi(s)\diamond Y'(s)ds.
\end{equation}
This definition is analogous to the representation of Skorohod integration with respect to Brownian motion $B$ in the Hida space setting. Let us investigate how the representation \eqref{eq:volterra-wick-def} may lead us to back to \eqref{eq:volterra-alos}.
 
Inserting for $Y'(s)$ and doing a Fubini-argument, lead to 
$$
\int_0^t\Phi(s)dY(s)=\int_0^t\mathcal K_g(\Phi)(t,s)\diamond W(s)ds=\int_0^t\mathcal K_g(\Phi)(t,s)\delta B(s)
$$
where 
$$
\mathcal K_g(f)(t,s)=f(s)g(s,s)+\int_s^tf(u)\partial_tg(u,s)du.
$$
If $g$ is not finite on the diagonal, and/or not differentiable in the first variable, we can do an integration-by-parts trick as in the paper by Alos {\it et al.} \cite{AMN}, which brings us back to the definition of $\mathcal K_g$ as in 
\eqref{def:K_g}. In conclusion, the stochastic Volterra integral by Alos {\it et al.} \cite{AMN} comes from a pathwise Lebesgue integration of the integrand and the Wick product with the noise process, being the time-derivative of the Volterra integrator.

An alternative heuristic definition based on Malliavin Calculus is done in
Barndorff-Nielsen {\it et al.} \cite{BNBPV}: using the classical and Malliavin integration-by-parts formulas along with the stochastic Fubini theorem, they are led to the following definition of the stochastic Volterra integral:
\begin{equation}
\label{def:strat-Volterra}
\int_0^t\Phi(s)d^{\circ}Y(s)=\int_0^t\mathcal K_g(\Phi)(t,s)\delta B(s)+\int_0^tD_s\mathcal K_g(\Phi)(t,s)ds
\end{equation}
Here, $D_s$ is the Malliavin derivative. This provides another definition of the stochastic integral, which can be cast in the Gelfand setting we have developed by introducing the operator
$$
\Xi_{t,s}(\Phi)=\mathcal K_g(\Phi)(t,s)\diamond W(s)+D_s\mathcal K_g(\Phi)(t,s)
$$
on the Hida distribution space. Care must be taken with the last term, possibly resorting to the extended Malliavin derivative. For more details on the stochastic Volterra integral in the White Noise context, we refer to Barndorff-Nielsen {\it et al.} \cite{BNBS}. Recalling the relationship between a Stratonovich-type of integral and Skorohod integration, as provided in \eqref{eq:relation-strat-skorohod}, we may coin 
the integral in \eqref{def:strat-Volterra} a Stratonovich-type of Volterra integral, while
\eqref{eq:volterra-alos} is of Ito-type.

\appendix

\section{The generalized Malliavin derivative}

We show here that the generalized Malliavin operator defined in
\eqref{def:gen-mall-der} is a derivative operator on $(\mathcal S)$:
\begin{proof}[Proof: $\mathcal D_f$ is a derivative operator:]
First, we notice that 
\begin{align*}
    \frac{d}{dx}\exp^{\diamond}(W(xf))\big\vert_{x=0}&=\frac{d}{dx}\exp\left(xW(f)-\frac12 x^2\vert f\vert_2^2\right)\big\vert_{x=0} \\
    &=\exp^{\diamond}(W(xf))\cdot(W(f)-x\vert f\vert_2^2)\big\vert_{x=0} \\
    &=W(f).
\end{align*}
Here we have used $\vert\cdot\vert_2$ to denote the $L^2$-norm on $\mathbb R$ equipped with the Lebesgue measure. Hence, by linearity of the Wick product,
\begin{align*}
    (\phi\cdot\psi)\diamond W(f)&=(\phi\cdot\psi)\diamond\frac{d}{dx}\exp^{\diamond}(W(xf))\big\vert_{x=0} \\
    &=\frac{d}{dx}\left((\phi\cdot\psi)\diamond\exp^{\diamond}(W(xf))\right)\big\vert_{x=0} \\
    &=\frac{d}{dx}\left(T_{-xf}(\phi\cdot\psi)\cdot\exp^{\diamond}(W(xf))\right)\big\vert_{x=0} \\
    &=\frac{d}{dx}\left(T_{-xf}\phi\cdot T_{-xf}\psi\cdot\exp^{\diamond}(W(xf))\right)\big\vert_{x=0} \\
    &=\left(\frac{d}{dx}T_{-xf}\phi\right)\big\vert_{x=0}\cdot\psi+\left(\frac{d}{dx}T_{-xf}\phi\right)\big\vert_{x=0}\cdot\phi \\
    &\qquad+\phi\cdot\psi\cdot\left(\frac{d}{dx}\exp^{\diamond}(W(xf))\right)\big\vert_{x=0} \\
    &=\left(\frac{d}{dx}T_{-xf}\phi\right)\big\vert_{x=0}\cdot\psi+\left(\frac{d}{dx}T_{-xf}\phi\right)\big\vert_{x=0}\cdot\phi+\phi\cdot\psi\cdot W(f) 
\end{align*}
We appealed to Holden {\it et al.} \cite[Thm. 2.10.6]{HOUZ} in the third equality above, denoting $T_g$ the translation operator along $g\in\mathcal S(\mathbb R)$. This yields that 
\begin{align*}
    \mathcal D_f(\phi\cdot\psi)&=\phi\cdot\psi\cdot W(f)-(\phi\cdot\psi)\diamond W(f) \\
    &=-\left(\frac{d}{dx}T_{-xf}\phi\right)\big\vert_{x=0}\cdot\psi-\left(\frac{d}{dx}T_{-xf}\phi\right)\big\vert_{x=0}\cdot\phi.
\end{align*}
But appealing to the adjoint translation operator $T^*_g$ and Holden {\it et al.} \cite[Thm. 2.10.3]{HOUZ}, it follows by linearity of the $\mathcal S$-transform,
\begin{align*}
    \mathcal S\left(\frac{d}{dx}T_{-xf}\phi\big\vert_{x=0}\right)(\xi)&=\frac{d}{dx}\langle T_{-xf}\phi,\exp^{\diamond}W(\xi)\rangle\big\vert_{x=0} \\
    &=\frac{d}{dx}\langle\phi,T_{-xf}^*\exp^{\diamond}(W(\xi))\rangle\big\vert_{x=0} \\
    &=\frac{d}{dx}\langle\phi,\exp^{\diamond}(W(\xi))\diamond\exp^{\diamond}(W(-xf))\rangle\big\vert_{x=0} \\
    &=-\langle\phi,W(f)\diamond\exp^{\diamond}(W(\xi))\rangle
\end{align*}
On the other hand, again by Thm. 2.10.6 in Holden {\it et al.} \cite{HOUZ}, 
\begin{align*}
\mathcal S(\phi\cdot W(f))(\xi)&=\langle\phi\cdot W(f),\exp^{\diamond}(W(\xi))\rangle \\
&=\langle\phi,W(f)\cdot\exp^{\diamond}(W(\xi))\rangle \\
&=\langle\phi,(W(f)-(f,\xi))\cdot\exp^{\diamond}(W(\xi))\rangle+(f,\xi)\langle\phi,\exp^{\diamond}(W(\xi))\rangle \\
&=\langle\phi,(T_{-\xi}W(f))\cdot\exp^{\diamond}(W(\xi))\rangle+\mathcal S(W(f)\diamond\phi)(\xi) \\
&=\langle\phi,W(f)\diamond\exp^{\diamond}(W(\xi))\rangle+\mathcal S(W(f)\diamond\phi)(\xi) \\
&=-\mathcal S\left(\frac{d}{dx}T_{-xf}\phi\big\vert_{x=0}\right)(\xi)+\mathcal S(W(f)\diamond\phi)(\xi)
\end{align*}
Therefore,
$$
\mathcal D_f\phi=-\frac{d}{dx}T_{-xf}\phi\big\vert_{x=0}
$$
Inserting this into the expression of $\mathcal D_f(\phi\cdot\psi)$ above, we find
$$
\mathcal D_f(\phi\cdot\psi)=(\mathcal D_f\phi)\cdot\psi+(\mathcal D_f\psi)\cdot\phi
$$
and the claimed derivative property holds. 
\end{proof}
Next follows the proof of the adjoint operator for the smoothed 
stochastic integral in \eqref{def:smoothed-wick-integral}:
\begin{proof}[Proof of Eq. \eqref{def:smoothed-wick-integral}:]
For the operator $\Xi_t(\Phi):=\Phi\diamond\widetilde{W}(t)$ we find by appealing to the adjoint translation operator and Holden {\it et al.} \cite[Thm. 2.10.3]{HOUZ},
\begin{align*}
    \Xi_t(\Phi)&=\Phi\diamond \widetilde{W}(t) \\
    &=\Phi\diamond\frac{d}{dx}\exp^{\diamond}(W(xf_t))\big\vert_{x=0} \\
    &=\frac{d}{dx}\left(\Phi\diamond\exp^{\diamond}(W(xf))\right)\big\vert_{x=0} \\
    &=\frac{d}{dx}\left(T_{xf_t}\Phi\right)\big\vert_{x=0}.
\end{align*}
Then we calculate for $\phi\in(\mathcal S)$,
\begin{align*}
    \langle\int_0^t\Phi(s)\delta\widetilde{B}(s),\phi\rangle&=\int_0^t\langle\frac{d}{dx}\left(T_{xf_s}^*\Phi(s)\right)\big\vert_{x=0},\phi\rangle ds \\
    &=\frac{d}{dx}\int_0^t\langle T_{xf_s}^*\Phi(s),\phi\rangle ds\big\vert_{x=0} \\
    &=\frac{d}{dx}\int_0^t\langle\Phi(s),T_{xf_s}\phi\rangle ds\big\vert_{x=0} \\
    &=\int_0^t\langle\Phi(s),\frac{d}{dx}\left(T_{xf_s}\phi\right)\big\vert_{x=0}\rangle ds
\end{align*}
From proof above, we have that 
$$
\frac{d}{dx}\left(T_{xf_s}\phi\right)\big\vert_{x=0}=\mathcal D_{f_s}\phi
$$
and the claim follows. 
\end{proof}

 \end{document}